\documentclass[12pt]{amsart}
\usepackage{amsmath}
\usepackage{amssymb}
\usepackage{amsfonts}

\newtheorem{theorem}{Theorem}
\theoremstyle{plain}

\newtheorem{corollary}{Corollary}

\newtheorem{lemma}{Lemma}

\newtheorem{proposition}{Proposition}
\newtheorem{remark}{Remark}

\numberwithin{equation}{section}

\begin{document}
\title[Isometric deformations of pseudoholomorphic curves in $\mathbb{S}^6$]{Isometric deformations of pseudoholomorphic curves in the
nearly K{\"a}hler sphere $\mathbb{S}^6$}
\author{Amalia-Sofia Tsouri}
\address{Department of Mathematics, University of Ioannina, 45110 Ioannina,
Greece}
\email{a.tsouri@uoi.gr}
\thanks{}

\begin{abstract}
The aim of the paper is to investigate the rigidity and the deformability
of pseudoholomorphic curves in the
nearly K{\"a}hler sphere $\mathbb{S}^6,$ among minimal surfaces in spheres.
Under various assumptions we describe the moduli space
of all noncongruent minimal surfaces $f\colon M\to\mathbb{S}^n$ that
are isometric to a pseudoholomorphic curve
in $\mathbb{S}^6.$ Moreover, we prove a Schur type theorem
(see \cite[p. 36]{Chnew}) for minimal surfaces in spheres. 
\end{abstract}

\maketitle

\section{Introduction}

Rigidity and deformability problems of a given isometric immersion are fundamental
problems of the theory of isometric immersions. Of particular interest is the
classification of all noncongruent minimal surfaces in a space form, that are isometric to a given one.
This problem was raised by Lawson in \cite{Law} and partial answers were provided by several authors.
For instance, see \cite{Cal2, J, L, Law,M1, N, S, S1, V9,V08}. 

The aforementioned problem has drown even more attention for minimal
surfaces in spheres. That is mainly due to the difficulty that arises from the fact that the Gauss map is
merely harmonic, in contrast to minimal surfaces in the Euclidean space where the Gauss map
is holomorphic. The classification problem of minimal surfaces in spheres that are
isometric to minimal surfaces in the sphere $\mathbb{S}^3$ was raised by Lawson
in \cite{L}, where he stated a conjecture that is still open. This conjecture has been
only confirmed for certain classes of minimal surfaces in spheres (see \cite{N, S, S1,
V9,V08}). It is worth noticing that a surface is locally isometric to a minimal surface
in $\mathbb{S}^3$ if its Gaussian curvature $K$ satisfies the spherical Ricci condition
$$
\Delta\log(1-K)=4K,
$$
away from totally geodesic points, where $\Delta$ is the Laplacian operator of the surface
with respect to its
induced metric.

In this paper, we turn our interest to a distinguished class of minimal surfaces in
spheres, the so-called {\textit{pseudoholomorphic curves}} in the nearly K{\"a}hler
sphere $\mathbb{S}^6.$ This class of surfaces was introduced by Bryant
\cite{Br} and has been widely studied (cf. \cite{BVW, H, EschVl}). The
pseudoholomorphic curves in $\mathbb{S}^6$ are nonconstant smooth maps from
a Riemann surface into the nearly K{\"a}hler sphere $\mathbb{S}^6,$ whose
differential is complex linear with respect to the almost complex structure of 
$\mathbb{S}^6$ that is induced from the multiplication of the Cayley numbers. 

In analogy with Calabi's work \cite{Cal2}, in the present paper we focus on the
following problem:

\begin{quotation}
\textit{Classify noncongruent minimal surfaces in spheres that are isometric to a given
pseudoholomorphic curve in the nearly K{\"a}hler sphere $\mathbb{S}^6.$}
\end{quotation}

One of the aims in this paper is to investigate the moduli space of all noncongruent substantial minimal
surfaces $f\colon M\to\mathbb{S}^n$ that are isometric to a given pseudoholomorphic curve
 $g\colon M\to\mathbb{S}^6.$ By substantial, we mean that $f(M)$ is not contained in any totally geodesic submanifold of $\mathbb{S}^n.$
It is known \cite{Br, EschVl} that any pseudoholomorphic curve $g\colon M\to
\mathbb{S}^6$ is $1$-isotropic (for the notion of $s$-isotropic surface see Section 2). The nontotally geodesic pseudoholomorphic curves
in $\mathbb{S}^6$ are either substantial in a totally geodesic $\mathbb{S}^5\subset
\mathbb{S}^6$ or substantial in $\mathbb{S}^6$  (see \cite{BVW}). In the latter case, the curve is either nonisotropic or
null torsion (studied by Bryant \cite{Br}). It turns out that null torsion curves are isotropic.
In order to study the above problem we have to deal
separately with these three classes of pseudoholomorphic curves. It is worth noticing that a characterization of Riemannian metrics
that arise as induced metrics on each class of these pseudoholomorphic curves was given in \cite{EschVl, V16} (for details see Section 5).

Flat minimal
surfaces in odd dimensional spheres (see \cite{K11, Br2}) are obviously isometric to any flat pseudoholomorphic curve in $\mathbb{S}^5.$ In \cite{VT} we provided a method to produce nonflat minimal surfaces in odd dimensional
spheres that are isometric to pseudoholomorphic curves in $\mathbb{S}^5.$ More precisely, let 
$g_{\theta}, 0\leq \theta<\pi$, be the associated family of a simply connected
pseudoholomorphic curve $g\colon M\to\mathbb{S}^{5}.$ We consider the surface
$G\colon M\to\mathbb{S}^{6m-1}$ defined by
\begin{equation}\label{gth}
G=a_{1}g_{\theta _{1}}\oplus \cdots\oplus a_{m}g_{\theta _{m}},  
\end{equation}
where $a_{1},\dots\,,a_{m}$ are any real numbers with $\sum_{j=1}^{m}a_{j}^
{2}=1,$ $0\leq \theta _{1}<\cdots<\theta_{m}<\pi,$ and $\oplus $ denotes the
orthogonal sum with respect to an orthogonal decomposition of the Euclidean space
$\mathbb{R}^{6m}.$  It is easy to see that $G$ is minimal and isometric to $g.$

It was verified in \cite{VT} that minimal surfaces given by \eqref{gth} belong to the class of
exceptional surfaces that was studied in \cite{V08, V16}. These are minimal surfaces
whose all Hopf differentials are holomorphic, or equivalently all curvature ellipses of
any order have constant eccentricity up to the last but one (see Sections 2 and 3 for details). In addition, in \cite{VT} it was proved
that minimal surfaces in spheres that are isometric to a given pseudoholomorphic curve in $\mathbb{S}^5$ are exceptional
under appropriate global assumptions. In fact, we proved that besides flat minimal
surfaces in odd dimensional spheres, the only simply connected exceptional surfaces
that are isometric to a pseudoholomorphic curve in $\mathbb{S}^5$ are of the type \eqref{gth}. 

Describing the moduli space of noncongruent minimal surfaces in spheres that are isometric
to a given pseudoholomorphic curve in the nearly K{\"a}hler $\mathbb{S}^6$ in full generality, turns out to be
a hard problem. To begin with, we investigate this moduli space in the class
of exceptional substantial surfaces in $\mathbb{S}^n.$ We denote by
$\mathcal{M}_n^{\mathrm{e}}(g)$ the moduli space of all noncongruent exceptional surfaces
$f\colon M\to\mathbb{S}^n$ that are isometric to a given pseudoholomorphic curve $g\colon M\to\mathbb{S}^6.$

At first we deal with nonflat pseudoholomorphic curves in
a totally geodesic $\mathbb{S}^5\subset\mathbb{S}^6$ in the case where $n$ is odd.
Given such a pseudoholomorphic curve $g,$ we are able to show that
the moduli space $\mathcal{M}_n^{\mathrm{e}}(g)$ is empty unless $n\equiv 5\;\mathrm{mod}\; 6,$
in which case $\mathcal{M}_n^{\mathrm{e}}(g)$ splits as
$$
\mathcal{M}_n^{\mathrm{e}}(g)=\mathbb{S}_\ast^{m-1}\times\Gamma_0,
$$
where $m=(n+1)/6,$

$$
\mathbb{S}_\ast^{m-1}=\Big\{(a_1,\dots,a_m)\in\mathbb{S}^{m-1}
\subset\mathbb{R}^m \colon \prod\limits_{j=1}^{m}a_j\neq0\Big\}
$$
and $\Gamma_0$ is a subset of
$$
\Gamma^m=\big\{(\theta_1,\dots,\theta_m)\in\mathbb{R}^m\colon 0\leq \theta _1<\cdots<\theta_m<\pi
\big\}.
$$
The case where $M$
is simply connected was studied in \cite[Theorem 3]{VT}, where it was proved that $\Gamma_0=\Gamma^m$.
In this paper, we prove that if $\Gamma_0$ is a proper subset of $\Gamma^m$ then it is
locally a disjoint finite
union of $d$-dimensional real analytic subvarieties where $d=0,\dots,m-1$.
If $M$ is compact and not homeomorphic to the torus, then it is shown that $\Gamma_0$ is
a proper subset of $\Gamma^m$ (see Theorems \ref{lineorfinite} and \ref{compactps}). As a
result, we are able to prove the following theorem, which provides an answer to
the aforementioned problem for minimal surfaces in spheres with low codimension. 
\begin{theorem}\label{mikradim}
Let $g\colon M\to\mathbb{S}^5$ be a compact pseudoholomorphic curve. If $M$ is not
homeomorphic to the torus, then the moduli space of all noncongruent substantial
minimal surfaces in $\mathbb{S}^n,
\, 4\le n\le 7,$ that are isometric to $g$ is empty, unless $n=5$ in which case the moduli space
is a finite set.
\end{theorem}
The necessity of the assumption
that the surface is not homeomorphic to the torus is justified by the class of flat
tori in $\mathbb{S}^5$ (see Remark \ref{flattori}).

Given a pseudoholomorphic curve 
$g\colon M\to\mathbb{S}^6,$ 
we are able to give the following description of the moduli space
(for the definition of the normal curvatures we refer the reader to Section 2).
\begin{theorem}\label{finiteorcircle}
Let $g\colon M\to \mathbb{S}^6$ be a pseudoholomorphic curve. The moduli space of all noncongruent minimal surfaces
$f\colon M\to \mathbb{S}^6$ that are isometric to $g$ and have the same normal curvatures with $g,$
is either a circle or a finite set.
\end{theorem}

 Isotropic pseudoholomorphic curves turn out to be rigid. For compact minimal surfaces our result is stated as follows.
\begin{theorem}\label{forintro}
Let $f\colon M\to\mathbb{S}^n$ be a compact substantial minimal surface. If $f$ is isometric to an isotropic pseudoholomorphic curve $g\colon M \to\mathbb{S}^6,$ then $n=6$ and $f$ is
congruent to $g.$
\end{theorem}
The same result holds if instead
of the compactness of the surface we assume that the surface is exceptional.

Finally, we deal with the third class of pseudoholomorphic curves in $\mathbb{S}^6,$
namely the nonisotropic ones. Under 
a global assumption on the Euler-Poincar\'{e} number of the second normal bundle (see Sections 2 and 3 for details),
we are able to prove the following result that provides a partial answer to our problem.

\begin{theorem}\label{compnon}
Let $g\colon M\to \mathbb{S}^6$ be a compact substantial pseudoholomorphic curve that is nonisotropic. If the Euler-Poincar\'{e} number of the second normal bundle of $g$ is nonzero,
then there are at most finitely many minimal surfaces in $\mathbb{S}^6$ isometric to
$g$ having the same normal curvatures with $g$. 
\end{theorem}

The necessity of the assumption on the codimension and the global assumptions in the above theorem is justified by the fact that the
direct sums of the associated family of a simply connected nonisotropic pseudoholomorphic
curve $g\colon M\to\mathbb{S}^6$ are isometric to $g$ (see Remark \ref{remarknoniso}).

In addition, we prove the following theorem that may be viewed as analogous to the classical
result of Schur (see \cite[p. 36]{Chnew}) in the realm
of minimal surfaces in spheres.

\begin{theorem}\label{teleutaioPhi}
Let $g\colon M\to\mathbb{S}^6$ be a compact, nonisotropic and substantial pseudoholomorphic curve and $\hat{g}\colon M\to\mathbb{S}^n$ be a
substantial minimal surface that is isometric to $g.$ 
If $\hat{g}$ is not $2$-isotropic and the second normal curvatures $K_2^\perp, \hat{K}_2^\perp$ of the surfaces $g$ and $\hat{g}$
respectively satisfy the inequality $\hat{K}_2^\perp\le K_2^\perp,$ then $n=6.$ 
Moreover, the moduli space of all such noncongruent minimal surfaces $\hat{g}\colon M\to\mathbb{S}^6$
that are isometric to $g,$
is either a circle or a finite set.
\end{theorem}

The paper is organized as follows: In Section 2, we fix the notation and give some
preliminaries. In Section 3, we recall the notion of Hopf differentials and some known results
about exceptional surfaces. In Section 4, we give some basic facts about absolute value type
functions, a notion that was introduced in \cite{EGT,ET} and will be exploited throughout the paper.
In Section 5, we recall some properties of
pseudoholomorphic curves in the nearly K{\"a}hler sphere $\mathbb{S}^{6}.$ 
 In Section 6, we investigate properties of the moduli space of noncongruent minimal surfaces, substantial in odd
dimensional spheres, that are isometric to a given
pseudoholomorphic
curve in $\mathbb{S}^{5}$ and give the proof of Theorem \ref{mikradim}. Section 7 is devoted to the case of isotropic pseudoholomorphic curves in $\mathbb{S}^{6}$ and give the proof of Theorem \ref{forintro}.
In the last section,  we deal with the study of nonisotropic
pseudoholomorphic curves in $\mathbb{S}^{6}$ and we give the proofs of Theorems \ref{finiteorcircle}, \ref{compnon} and  \ref{teleutaioPhi}. 

\section{Preliminaries}

In this section, we collect several facts and definitions about minimal surfaces in spheres.
For more details we refer to \cite{DF} and \cite{DV2k15}. 
\vspace{1,5ex}

Let $f\colon M\to\mathbb{S}^n$ be an isometric immersion of a $2$-dimensional
Riemannian manifold. The $k^{th}$\emph{-normal space} of $f$ at $p\in M$ for
$k\geq 1$ is defined as
$$
N^f_k(p)={\rm span}\left\{\alpha^f_{k+1}(X_1,\ldots,X_{k+1}):X_1,\ldots,X_{k+1}\in
T_pM\right\},
$$
where the symmetric tensor 
$$
\alpha^f_s\colon TM\times\cdots\times TM\to N_fM,\;\; s\geq 3, 
$$
given
inductively by
$$
\alpha^f_s(X_1,\ldots,X_s)=\left(\nabla^\perp_{X_s}\cdots\nabla^\perp_{X_3}
\alpha^f(X_2,X_1)\right)^\perp,
$$
is called the $s^{th}$\emph{-fundamental form} and $\alpha^f\colon TM\times TM\to N_fM$ stands for the standard second fundamental 
form of $f$ with values in the normal bundle. Here, $\nabla^{\perp}$ denotes the induced
connection in the normal bundle $N_fM$ of $f$ and $(\,\cdot\,)^\perp$ stands for the
projection onto the orthogonal complement of $N^f_1\oplus\cdots\oplus N^f_{s-2}$ in
$N_fM.$ If $f$ is minimal, then ${\rm dim}N^f_k(p)\le2$ for all $k\ge1$ and any $p\in M$ (cf. \cite{DF}).

A surface $f\colon M\to\mathbb{S}^n$ is called \emph{regular} if for each $k$ the
subspaces $N^f_k$ have constant dimension and thus form normal subbundles. Notice
that regularity is always verified  along connected components of an open dense subset of
$M.$ 

Assume that an immersion $f\colon M\to\mathbb{S}^n$ is minimal and substantial.
By the latter, we mean that $f(M)$ is not contained in any totally geodesic submanifold of $\mathbb{S}^n.$ In this case, the normal bundle
of $f$ splits along an open dense subset of $M$ as
$$
N_fM=N_1^f\oplus N_2^f\oplus\dots\oplus N_m^f,\;\;\; m=[(n-1)/2],
$$
since all higher normal bundles  have rank two except possible the last one that has rank
one if $n$ is odd; see \cite{Ch} or \cite{DF}. Moreover, if $M$ is oriented, then an
orientation is induced on each plane subbundle $N_s^f$  given by the ordered basis
$$
\alpha^f_{s+1}(X,\ldots,X),\;\;\;\alpha^f_{s+1}(JX,\ldots,X),
$$
where $0\neq X\in TM,$ and $J$ is the complex structure determined by the orientation and the metric.

If $f\colon M\to\mathbb{S}^n$ is a minimal surface, then at any point $p\in M$ and for each
$N_r^f$, $1\leq r\leq m$, the \emph{$r^{th}$-order curvature ellipse}
$\mathcal{E}^f_r(p)\subset N^f_r(p)$ is defined  by
$$
\mathcal{E}^f_r(p) = \left\{\alpha^f_{r+1}(Z^{\varphi},\ldots,Z^{\varphi})\colon\,
Z^{\varphi}=\cos\varphi Z+\sin\varphi JZ\;\mbox{and}\;\varphi\in[0,2\pi)\right\},
$$
where $Z\in T_xM$ is any vector of unit length.

A substantial regular surface $f\colon M\to\mathbb{S}^{n}$ is called
\emph{$s$-isotropic} if it is  minimal and  at any point $p\in M$ the  curvature ellipses
$\mathcal{E}^f_r(p)$ contained in all two-dimensional  $N^f_r$$\,{}^{\prime}$s 
are circles for any $1\le r\le s.$ It is called \emph{isotropic} if it is $s$-isotropic for any $s.$

The $r$-th \textit{normal curvature} $K_{r}^{\perp}$ of $f$ is defined by
\begin{equation*}
K_{r}^{\perp}={\frac{2}{\pi}}{\hbox {Area}}(\mathcal{E}^f\sb r).
\end{equation*}
If $\kappa _{r}\ge\mu _{r}\ge0$ denote the length of the semi-axes of the  curvature ellipse
$\mathcal{E}^f\sb r,$ then
\begin{equation}\label{elipsi}
K_{r}^{\perp}=2\kappa _{r}\mu _{r}.
\end{equation}
Clearly, the curvature ellipse $\mathcal{E}^f\sb r(p)$ at a point $p\in M$  is a circle if and only if $\kappa _{r}(p)= \mu_{r}(p).$

The \textit{eccentricity} $\varepsilon\sb r$ of the  curvature ellipse $\mathcal{E}^f\sb r$ is given by
\begin{equation*}
\varepsilon\sb r=\frac{\left(\kappa^{2}_{r}-\mu^{2}_{r}\right)^{1/2}}{\kappa_{r}},
\end{equation*}
where $\left(\kappa^{2}_{r}-\mu^{2}_{r}\right)^{1/2}$ is the distance from the center to a focus,
and can be thought of as a measure of how far $\mathcal{E}^f\sb r$ deviates from being a
circle.

The $a$-\textit{invariants} (see \cite{V16}) are the functions
\begin{equation*}
a^{\pm}_{r}= \kappa _{r}{\pm}  \mu _{r}= \left(2^{-r} \Vert \alpha^f_{r+1}
\Vert^{2} \pm K_{r}^{\perp}\right)^{1/2}.
\end{equation*}
These functions determine the geometry of the $r$-th curvature ellipse.

Denote by $\tau^o_f$ the index of the last plane bundle, in the orthogonal decomposition
of the normal bundle. Let $\{e_1,e_2\}$ be a local tangent orthonormal frame and
$\{e_\alpha\}$ be a local orthonormal frame of the normal bundle such that
$\{e_{2r+1},e_{2r+2}\}$ span $N_r^f$ for any $1\le r\le\tau^o_f$ and $e_{2m+1}$
spans the line bundle $N^f_{m+1}$ if $n=2m+1.$ For any $\alpha=2r+1$ or $\alpha=2r+2,$ we set 
$$
h_1^{\alpha}=\langle
\alpha^f_{r+1}(e_1,\dots,e_1),e_{\alpha}\rangle,{\ }h_2^{\alpha}=\langle
\alpha^f_{r+1}(e_1,\dots,e_1,e_2),e_{\alpha}\rangle,
$$
where $\langle\cdot,\cdot\rangle$ is the standard metric of $\mathbb{S}^n.$ Introducing the complex valued functions
$$
H_{\alpha}=h_1^{\alpha}+ih_2^{\alpha}\;\;\text{for any}\;\;\alpha=2r+1\;\;\text{or}\;\;\alpha=2r+2,
$$
it is not hard to verify that the $r$-th normal
curvature is given by
\begin{equation}\label{prwthsxeshden}
K_r^{\perp}=i\left(H_{2r+1}{\overline{H}_{2r+2}}-{\overline{H}_{2r+1}}
H_{2r+2}\right).
\end{equation}

The length of the $(r+1)$-th fundamental form $\alpha^f_{r+1}$ is given by
\begin{equation}\label{deuterhsxeshden}
\Vert \alpha^f_{r+1}\Vert ^2=2^r\big(|{H_{2r+1}}|^2+|{H_{2r+2}}|
^2\big),
\end{equation}
or equivalently (cf. \cite{A})
\begin{equation}\label{si}
\Vert \alpha^f_{r+1}\Vert ^2=2^r(\kappa_r^2+\mu_r^2).
\end{equation}
In particular, it follows from the Gauss equation that 
\begin{equation}\label{Gausseqa2}
\Vert \alpha^f_{2}\Vert ^2=2(1-K).
\end{equation}

\smallskip
Each plane
subbundle $N_r^f$ inherits a Riemannian connection from that of the normal bundle. Its
\textit{intrinsic curvature} $K^*_r$ is given by the following proposition (cf. \cite{A}).

\begin{proposition}\label{5}
\textit{The intrinsic curvature $K_r^{\ast}$ of each plane subbundle $N_{r}^f$
of a minimal surface} $f\colon M\to \mathbb{S}^{n}$ \textit{is given by} 
\begin{equation*}
K_{1}^{\ast}=K_1^{\perp}-{\frac{\Vert \alpha^f_3\Vert ^2}{2K_1^{\perp}}} 
\;\; \text{and}\;\; K_r^{\ast}={\frac{K_r^{\perp}}{(K_{r-1}^{\perp})^2}}{\frac{
\Vert \alpha^f_{r}\Vert ^{2}}{2^{r-2}}}-{\frac{\Vert \alpha^f_{r+2}\Vert ^2}{
2^rK_r^{\perp}}}\;\;\text{for}\;\;2\leq r\leq \tau_f^o.
\end{equation*}
\end{proposition}

Let $f\colon M\to\mathbb{S}^n$ be a minimal isometric immersion. If $M$ is
simply connected, there exists a one-parameter \emph{associated family} of minimal
isometric immersions $f_\theta\colon M\to\mathbb{S}^n,$ where  $\theta\in\mathbb{S}^1=[0,\pi).$ To see this, for each
$\theta\in\mathbb{S}^1$ consider the orthogonal parallel tensor field 
$$
J_{\theta}=\cos\theta I+\sin\theta J,
$$
where $I$ is the identity endomorphism of the tangent bundle and $J$ is the complex
structure of $M$ induced by
the metric and the orientation.  Then, the symmetric section $\alpha^f(J_\theta\cdot,
\cdot)$ of the bundle $\text{Hom}(TM\times TM,N_f M)$ satisfies the Gauss, Codazzi and
Ricci equations, with respect to the same normal connection; see \cite{DG2} for details. 
Therefore, there exists a minimal isometric immersion  $f_{\theta}\colon M\to
\mathbb{S}^n$ whose second fundamental form is given by
\begin{equation*}
\alpha^{f_{\theta}}(X,Y)=T_\theta\alpha^f(J_{\theta}X,Y),
\end{equation*}
where $T_\theta\colon N_fM\to N_{f_{\theta}}M$ is a parallel 
vector bundle isometry that identifies the normal subspaces
$N_s^f$ with $N_s^{f_\theta}$, $s\geq 1.$

\section{Hopf differentials and Exceptional surfaces}

Let $f\colon M\to\mathbb{S}^n$ be a minimal surface. The complexified tangent bundle
$TM\otimes \mathbb{C}$ is decomposed into the eigenspaces $T^{\prime}M$ and $T^{\prime \prime}M$ of the complex structure
$J$,  corresponding to the
eigenvalues $i$ and $-i.$  The $(r+1)$-th fundamental form $\alpha^f_{r+1}$, which takes values in
the normal subbundle $N_{r}^f$, can be complex linearly extended to $TM\otimes
\mathbb{C}$ with values in the complexified vector bundle $N_{r}^f\otimes \mathbb{C}$
and then decomposed into its $(p,q)$-components, $p+q=r+1,$ which are tensor
products of $p$ differential 1-forms vanishing on $T^{\prime \prime}M$ and $q$ differential 1-forms vanishing
on $T^{\prime}M.$ The minimality of $f$ is equivalent to the vanishing of the
$(1,1)$-part of the second fundamental form. Hence, the $(p,q)$-components of $\alpha^f_{r+1}$
vanish unless $p=r+1$ or $p=0,$ and consequently for a local complex coordinate $z$
on $M$, we have the following decomposition 
\begin{equation*}
\alpha^f_{r+1}=\alpha_{r+1}^{(r+1,0)}dz^{r+1}+\alpha_{r+1}^{(0,r+1)}d\bar{z}^
{r+1},
\end{equation*}
where 
\begin{equation*}
\alpha_{r+1}^{(r+1,0)}=\alpha^f_{r+1}(\partial,\dots,\partial),\;\;\alpha_{r+1}^{(0,r+1)}=\overline{\alpha_{r+1}^{(r+1,0)}}\;\;\;\text{and}\;\;\;\partial ={\frac{1}{2}}\big({\frac{\partial}{\partial x}}-i{\frac{\partial}{\partial y}}\big).
\end{equation*}

The \textit{Hopf differentials} are the differential forms (see \cite{V})
\begin{equation*}
\Phi _{r}=\langle \alpha_{r+1}^{(r+1,0)},\alpha_{r+1}^{(r+1,0)}\rangle dz^{2r+2}
\end{equation*}
of type $(2r+2,0),r=1,\dots,[(n-1)/2],$ where $\langle \cdot,\cdot\rangle$ denotes the extension of the usual Riemannian metric
of $\mathbb{S}^n$ to a complex bilinear form. These forms are
defined on the open subset where  the minimal surface is regular and are independent of the choice of coordinates, while
$\Phi _{1}$ is globally well defined.

Let $\{e_1,e_2\}$ be a local orthonormal frame in the tangent bundle.
It will be convenient to use complex vectors, and we put
\begin{equation*}
\text{ }E=e_1-ie_2\;\;
\text{and}\;\; \phi =\omega _{1}+i\omega _2,
\end{equation*}
where $\{\omega_1,\omega_2\}$ is the dual frame. We choose a local complex
coordinate $z=x+iy$ such that $\phi =Fdz.$

From the definition of Hopf differentials, we easily obtain 
\begin{equation*}
\Phi _{r}={\frac{1}{4}}\left({\overline{H}_{2r+1}^2}+{\overline{H}_{2r+2}^2}
\right) \phi^{2r+2}.
\end{equation*}

Moreover, using (\ref{prwthsxeshden}) and (\ref{deuterhsxeshden}), we find that
\begin{equation}\label{what}
\left\vert \langle \alpha_{r+1}^{(r+1,0)},\alpha_{r+1}^{(r+1,0)}\rangle \right\vert
^2=\frac{F^{2r+2}}{2^{2r+4}}\left(\Vert \alpha^f_{r+1}\Vert
^4-4^r(K_r^{\perp})^2\right).
\end{equation}
Thus, the zeros of $\Phi _r$ are precisely the points where the $r$-th curvature ellipse $
\mathcal{E}^f\sb r$ is a circle. Moreover, using \eqref{elipsi} and \eqref{si} we obtain the following: 
 \begin{lemma}\label{systematic}
Let $f\colon M\to\mathbb{S}^{n}$ be a minimal surface. Then the following assertions are equivalent:

(i) The surface $f$ is $s$-isotropic.

(ii) The Hopf differentials satisfy $\Phi_r=0$ for any $1\le r\le s.$

(iii) The length of the $(r+1)$-th fundamental form $\alpha^f_{r+1}$ and the $r$-th normal curvature $K_{r}^{\perp}$
satisfy
\begin{equation*}
\Vert \alpha^f_{r+1}\Vert ^2=2^rK_{r}^{\perp},
\end{equation*}
for any $1\le r\le s.$ In particular, the surface $f$ is 1-isotropic if and only if
the first normal curvature $K_{1}^{\perp}$ satisfies
\begin{equation*}
K_{1}^{\perp}=1-K.
\end{equation*}

\end{lemma}

The Codazzi equation implies that $\Phi _{1}$ is always holomorphic (cf. \cite{Ch,ChW}).
Besides $\Phi_1$, the rest Hopf differentials are not always holomorphic. The following
characterization of the holomorphicity of Hopf differentials was given in \cite{V08}, in
terms of the eccentricity of curvature ellipses of higher order. 

\begin{theorem}\label{ena}
Let $f\colon M\to \mathbb{S}^n$ be a minimal surface. Its Hopf differentials
$\Phi _{2},\dots,\Phi_{r+1}$ are holomorphic if and only if the higher curvature ellipses
have constant eccentricity up to order $r.$
\end{theorem}

A minimal surface in $\mathbb{S}^n$ is called \textit{$r$-exceptional} if all Hopf
differentials up to order $r+1$ are holomorphic, or equivalently if all higher curvature
ellipses up to order $r$ have constant eccentricity. A minimal surface in $\mathbb{S}^n$
is called \textit{exceptional} if it is $r$-exceptional for $r=[(n-1)/2-1].$ 
This class of minimal surfaces may be viewed as the next simplest to superconformal ones.
In fact, superconformal minimal surfaces are indeed exceptional, characterized by the fact that all
Hopf differentials vanish up to the
last but one, which is equivalent to the fact that all higher curvature ellipses are circles up
to the last but one. As a matter of fact, there is an abundance of exceptional surfaces.

We recall some results for exceptional surfaces proved in \cite{V08}, that will be used in
the proofs of our main results.

\begin{proposition}\label{3i}
Let $f\colon M\to\mathbb{S}^n$ be an $(r-1)$-exceptional surface. At regular
points the following hold:

(i) For any $1\leq s\leq r-1,$ we have
\begin{equation*}
\Delta \log \left\Vert \alpha_{s+1}\right\Vert ^2=2\big((s+1)K-K_s^{\ast}\big),
\end{equation*}
where $\Delta $ is the Laplacian operator with respect to the induced metric $ds^{2}.$

(ii) If $\Phi _{r}\neq 0$\textit{, then}
\begin{equation*}
\Delta \log \left(\left\Vert \alpha_{r+1}\right\Vert^2+2^rK_r^{\perp}\right) =2\big((r+1)K-K_r^{\ast}\big)
\end{equation*}
and 
\begin{equation*}
\Delta \log \left(\left\Vert \alpha_{r+1}\right\Vert^2-2^rK_r^{\perp}\right) =2\big((r+1)K+K_r^{\ast}\big).
\end{equation*}

(iii) If $\Phi _{r}=0$\textit{, then}
\begin{equation*}
\Delta \log \left\Vert \alpha_{r+1}\right\Vert^2=2\big((r+1)K-K_r^{\ast}\big).
\end{equation*}

(iv) The intrinsic curvature of the $s$-th
normal bundle $N_s^f$\textit{\ is} $K_{s}^{\ast}=0$ 
if $1\leq s\leq r-1$ and $\Phi _s\neq 0.$
\end{proposition}

A remarkable property of exceptional surfaces is that singularities of the higher normal bundles
are of holomorphic type and can be smoothly extended to vector bundles. This fact was proved in
\cite[Proposition 4]{V08}.
\begin{proposition}\label{neoksanaafththfora}
Let $f\colon M\to \mathbb{S}^n$ be an $r$-exceptional surface. Then the set $L_0$,
where $f$ fails to be regular, consists of isolated points and all $N_s^f$'s and the Hopf
differentials $\Phi_s$'s extend smoothly to $L_0$ for any $1\leq s\leq r.$ 
\end{proposition}

\section{Absolute value type functions}

For the proof of our results, we shall use the notion of absolute value type functions
introduced in \cite{EGT,ET}. A smooth complex valued function $p$ defined on a
Riemann surface is called of \textit{holomorphic type} if locally $p=p_0p_1,$ where $p_0$
is holomorphic and $p_1$ is smooth without zeros. A function $u\colon M\to\lbrack 0,+
\infty)$ defined on a Riemann surface $M$ is called of \textit{absolute value type} if there
is a function $p$ of holomorphic type on $M$ such that $u=|p|.$

The zero set of such a function on a connected compact oriented surface $M$ is either
isolated or the whole of $M$, and outside its zeros the function is smooth. If $u$ is a
nonzero absolute value type function, i.e., locally $u=|t_0|u_1$, with $t_0$ holomorphic,
the order $k\ge1$ of any point $p\in M$ with $u(p)=0$ is the order of $t_0$ at $p.$ Let
$N(u)$ be the sum of the orders for all zeros of $u.$ Then $\Delta\log u$ is bounded on
$M\smallsetminus\left\{u=0\right\}$ and its integral is computed in the following lemma
that was proved in \cite{EGT,ET}.

\begin{lemma}\label{forglobal}
Let $(M,ds^2)$ be a compact oriented two-dimensional Riemannian manifold with area
element $dA.$

(i) If $u$ is an absolute value type function on $M,$ then 
\begin{equation*}
\int_{M}\Delta \log udA=-2\pi N(u).
\end{equation*}

(ii) If $\Phi $ is a holomorphic symmetric $(r,0)$-form on $M,$ then either $\Phi =0$ or
$N(\Phi)=-r\chi (M),$ where $\chi (M)$ is the Euler-Poincar\'{e} characteristic of $M.$
\end{lemma}

The following lemma, that was proved in \cite{N}, provides a sufficient condition for a
function to be of absolute value type.
\begin{lemma}\label{dena}
Let $D$ be a plane domain containing the origin with coordinate $z$ and $u$ be a real
analytic nonnegative function on $D$ such that $u(0)=0.$ If $u$ is not identically zero
and $\log u$ is harmonic away from the points where $u=0$, then $u$ is of absolute
value type and the order of the zero of $u$ at the origin is even.
\end{lemma}

\section{Pseudoholomorphic curves in $\mathbb{S}^6$}

In this section we summarize some well known facts about pseudoholomorphic
curves in the nearly K{\"a}hler sphere $\mathbb{S}^6$. It is known that the
multiplicative structure on the Cayley numbers $\mathbb{O}$ can be
used to define an almost complex structure on the sphere $\mathbb{S}^6$ in
$\mathbb{R}^7.$ This almost complex structure is not integrable but it is nearly
K{\"a}hler. A \textit{pseudoholomorphic curve}, which was introduced by Bryant \cite{Br}, is a nonconstant smooth map
$g\colon M\to\mathbb{S}^6$ from a Riemann surface $M$ into the nearly K{\"a}hler
sphere $\mathbb{S}^6,$ whose differential is complex linear.

It is known \cite{Br, EschVl} that any pseudoholomorphic curve $g\colon M\to\mathbb{S}
^6$ is $1$-isotropic. The nontotally geodesic pseudoholomorphic curves in $\mathbb{S}^6$
are are either substantial in a totally geodesic $\mathbb{S}^5\subset
\mathbb{S}^6$ or substantial in $\mathbb{S}^6$  (see \cite{BVW}). In the latter case, the curve is either
null torsion (studied by Bryant \cite{Br}) or nonisotropic. It turns out that null torsion curves are isotropic. 

The following theorem \cite{EschVl} provides a characterization of Riemannian metrics
that arise as induced metrics on pseudoholomorphic curves in $\mathbb{S}^5.$
\begin{theorem}\label{eschvlach}
Let $(M,ds^2)$ be a simply connected Riemann surface, with Gaussian curvature $K\leq 1$ and Laplacian operator
$\Delta$. Suppose that the function $1-K$ is of absolute value type. Then there exists an isometric pseudoholomorphic curve $g\colon M\to\mathbb{S}^5$ if and only if
\[
\Delta\log(1-K)=6K.\tag{$\ast$}
\]
In fact, up to translations with elements of $G_2$, that is the set $Aut(\mathbb{O})\subset SO(7),$ there is precisely one associated family of such maps.
\end{theorem}
The above result shows that a minimal surface in a sphere is locally isometric to a
pseudoholomorphic curve in $\mathbb{S}^5$ if its Gaussian curvature satisfies the 
condition $(\ast)$ at points where $K<1$ or equivalently if the metric
$d\hat{s}^2=(1-K)^{1/3}ds^2$ is flat.

Let $g\colon M\to\mathbb{S}^5$ be a pseudoholomorphic curve and let
$\xi\in N_fM$ be a smooth unit vector field that spans the extended line bundle $N_2^g$ over the isolated set of points where $f$ fails to be regular  (see Proposition \ref{neoksanaafththfora}).   The
surface $g^*\colon M\to\mathbb{S}^5$ defined by $g^*=\xi$ is called the \textit{polar surface} of $g.$ It has been proved in \cite[Corollary 3]{V16} that the surfaces $g$ and $g^*$ are
congruent.

We recall the following theorem \cite{EschVl}, which provides a characterization of Riemannian metrics
that arise as induced metrics on isotropic substantial pseudoholomorphic curves in $\mathbb{S}^6.$

\begin{theorem}
Let $(M,ds^2)$ be a simply connected Riemann surface, with Gaussian curvature $K\leq 1$ and Laplacian operator
$\Delta$. Suppose that the function $1-K$ is of absolute value type. Then there exists an isotropic pseudoholomorphic curve $g\colon M\to\mathbb{S}^6,$ unique up to translations with elements of $G_2,$ if and only if
\[
\Delta\log(1-K)=6K-1.\tag{$\ast\ast$}
\]
\end{theorem}

The following theorem \cite{V16} provides a characterization of Riemannian metrics
that arise as induced metrics on nonisotropic substantial pseudoholomorphic curves in $\mathbb{S}^6.$
\begin{theorem}
Let $(M,ds^2)$ be a simply connected Riemann surface, with Gaussian curvature $K\leq 1$ and Laplacian operator
$\Delta$. Suppose that the function $1-K$ is of absolute value type. Then there exists a nonisotropic pseudoholomorphic curve $g\colon M\to\mathbb{S}^6,$ unique up to translations with elements of $G_2,$ if and only if
\[
\Delta\log\left((1-K)^2\left(1-6K+\Delta\log\left(1-K\right)\right)\right)=12K.
\]
Moreover the following holds:
\begin{equation}\label{trik}
6K-1<\Delta\log(1-K)<6K.
\end{equation}
\end{theorem}

\section{Isometric deformations of pseudoholomorphic curves in $\mathbb{S}^5$}
We are interested in nontrivial isometric deformations of pseudoholomorphic
curves in $\mathbb{S}^5.$ Given a pseudoholomorphic curve
$g\colon M\to\mathbb{S}^5,$ we would like to describe the moduli space of all
noncongruent substantial minimal surfaces $f\colon M\to\mathbb{S}^n$ that are locally
isometric to the curve $g.$ For the class of the exceptional surfaces we denote the
above mentioned space by $\mathcal{M}_n^{\mathrm{e}}(g).$ Hereafter we assume
that $n$ is odd and $M$ is nonflat.

If $M$ is simply connected, it has been proved in \cite[Theorem 3]{VT} that $n\equiv 5\;
\mathrm{mod}\; 6,$ and

$$
\mathcal{M}_n^{\mathrm{e}}(g)=\mathbb{S}_\ast^{m-1}\times\Gamma^m,
$$
where $m=(n+1)/6,$
$$
\mathbb{S}_\ast^{m-1}=\Big\{\bold{a}=(a_1,\dots,a_m)\in\mathbb{S}^{m-1}
\subset\mathbb{R}^m \colon \prod\limits_{j=1}^{m}a_j\neq0\Big\}
$$
and 
$$
\Gamma^m=\big\{{\pmb{\theta}}=(\theta_1,\dots,\theta_m)\in[0,\pi)
\times\cdots\times[0,\pi)\colon 0\leq \theta _1<\cdots<\theta_m<\pi
\big\}.
$$

Our aim in this section is to study the moduli space of noncongruent isometric deformations of a
nonsimply connected pseudoholomorphic curve $g\colon M\to\mathbb{S}^5.$ We
consider the covering map $\Pi\colon\tilde{M}\to M,$
$\tilde{M}$ being the universal cover of $M$ with the metric and orientation that
make $\Pi$ an orientation preserving local isometry. Corresponding objects on
$\tilde{M}$ are denoted with tilde. Then the map $\tilde{g}\colon \tilde{M}\to
\mathbb{S}^5$ with $\tilde{g}=g\circ\Pi$ is a pseudoholomorphic curve. Obviously,
since $\tilde{g}$ is simply connected, we know from \cite[Theorem 3]{VT} that

$$
\mathcal{M}_n^{\mathrm{e}}(\tilde{g})=\mathbb{S}_\ast^{m-1}\times
\Gamma^m.
$$

For any $(\bold{a},\pmb{\theta})\in\mathbb{S}_\ast^{m-1}\times
\bar{\Gamma}^m,$ where $\bar{\Gamma}^{m}$ is the closure of ${\Gamma}^m,$ we consider
the minimal surface 
$\tilde{g}_{\bold{a},\pmb{\theta}}\colon \tilde{M}\to\mathbb{S}^
{6m-1}\subset\mathbb{R}^{6m}$ defined by
\[
\tilde{g}_{\bold{a},{\pmb{\theta}}}=a_1\tilde{g}_{\theta_{1}}\oplus\cdots
\oplus a_m\tilde{g}_{\theta_m},
\]
where $\oplus$
denotes the orthogonal sum with respect to an orthogonal decomposition of
$\mathbb{R}^{6m}.$ Each surface $\tilde{g}_{\theta_{j}}\colon \tilde{M}\to\mathbb{S}
^5, \,j=1,\dots,m,$ is a member of the associated family of $\tilde{g}.$

Clearly, given an exceptional surface $f\colon M\to\mathbb{S}^n
$ in the moduli space of the curve $g,$ the minimal surface $\tilde
{f}\colon \tilde{M}\to\mathbb{S}^n$ with $\tilde{f}=f\circ\Pi$ belongs to the moduli
space $\mathcal{M}_n^{\mathrm{e}}(\tilde{g})$ of the curve
$\tilde{g}.$ Therefore, the moduli space $\mathcal{M}_n^{\mathrm{e}}(g)$ can
be described as the subset of all $(\bold{a},\pmb{\theta})$ in $\mathcal{M}^{\mathrm
{e}}_n(\tilde{g})$ such that $\tilde{g}_{\bold{a},\pmb{\theta}}$ factors as
$F\circ\Pi$ for some exceptional surface $F\colon M \to\mathbb{S}^n.$
We follow this notation throughout this section.

The group $\mathcal{D}$ of \textit{deck transformations} of
the covering map $\Pi\colon \tilde{M}\to M$ consists of all diffeomorphisms
$\sigma\colon\tilde{M}\to\tilde{M}$ such that $\Pi\circ\sigma=\Pi.$

We need the following lemmas.
\begin{lemma}\label{deckbig}
For each $\sigma\in\mathcal{D}$ the surfaces $\tilde{g}_{\bold{a},\pmb{\theta}}$
and $\tilde{g}_{\bold{a},\pmb{\theta}}\circ\sigma$ are congruent for every $(\bold
{a},\pmb{\theta})\in\mathbb{S}_\ast^{m-1}\times
\bar{\Gamma}^m,$ that is there exists $\Phi_{\pmb{\theta}}(\sigma)\in\mathrm{O}
(n+1)$ such that
\begin{equation*}
\tilde{g}_
{\bold{a},{\pmb{\theta}}}\circ\sigma=\Phi_{\pmb{\theta}}(\sigma)\circ\tilde{g}_
{\bold{a},{\pmb{\theta}}}.
\end{equation*}
\end{lemma}
\begin{proof}
It follows from \cite[Proposition 9]{DV} that the surfaces $\tilde{g}_{\theta}$ and $\tilde{g}_
{\theta}\circ\sigma$ are congruent for all $\theta\in[0,\pi).$ Therefore, there exists
$\Psi_\theta(\sigma)\in\mathrm{O}(7)$ such that 
\begin{equation} \label{decksmall}
\tilde{g}_\theta\circ\sigma=
\Psi_\theta(\sigma)\circ\tilde{g}_\theta 
\end{equation}
for every $\theta\in[0,\pi).$

We define the isometry $\Phi_{\pmb{\theta}}(\sigma)\in\mathrm{O}(n+1)$ given by 
\[
\Phi_{\pmb{\theta}}(\sigma)=\Psi_{\theta_1}(\sigma)\oplus\cdots\oplus\Psi_
{\theta_m}(\sigma),
\]
with respect to an orthogonal decomposition $\mathbb{R}^{6m}=\mathbb{R}^{6}
\oplus\cdots\oplus\mathbb{R}^{6}.$
That 
\begin{equation*}
\tilde{g}_
{\bold{a},{\pmb{\theta}}}\circ\sigma=\Phi_{\pmb{\theta}}(\sigma)\circ\tilde{g}_
{\bold{a},{\pmb{\theta}}}
\end{equation*}
holds, follows directly from \eqref{decksmall}.
\end{proof}

\begin{remark}\label{remark1}
The isometry $\Phi_{\pmb{\theta}}(\sigma)$ is real analytic with respect to
${\pmb{\theta}}$ (cf. \cite{EQ}).
\end{remark}
\begin{lemma}\label{lemmaphitheta}
If $(\bold{a},\pmb{\theta})$ belongs to $\mathcal{M}_n^{\mathrm{e}}
(\tilde{g}),$ then $(\bold{a},\pmb {\theta})$ belongs to $\mathcal{M}^{\mathrm
{e}}_n(g)$ if and only if
\begin{equation}\label{phitheta}
\Phi_{\pmb{\theta}}(\mathcal{D})=\left\{\mathrm{Id}\right\}.
\end{equation} 
\end{lemma}
\begin{proof}
Let $(\bold{a},\pmb{\theta})\in\mathcal{M}^{\mathrm{e}}_
n(g).$ There exists an exceptional surface $F\colon M\to\mathbb
{S}^{n}$ such that 
\begin{equation*}
F\circ\pi=\tilde{g}_{\bold{a},{\pmb{\theta}}}.
\end{equation*} 
Composing with an arbitrary $\sigma\in\mathcal{D}$ and using 
Lemma \ref{deckbig}, we obtain
\begin{equation*}
\tilde{g}_{\bold{a},{\pmb{\theta}}}=\Phi_{\pmb{\theta}}(\sigma)\circ\tilde{g}_
{\bold{a},{\pmb{\theta}}}.
\end{equation*}
The fact that $\tilde{g}_{\bold{a},{\pmb{\theta}}}$ has substantial
codimension yields \eqref{phitheta}.

Conversely, assume that \eqref{phitheta} holds. We will prove that $\tilde{g}_
{\bold{a},{\pmb{\theta}}}$ factors as $F\circ\Pi$ where $F\colon M\to\mathbb{S}
^{n}$ is an exceptional surface. At first we claim that $\tilde{g}_{\bold{a},{\pmb
{\theta}}}$ remains constant on each fiber of the covering map $\Pi.$ Indeed, let
$\tilde{p}_1, \tilde{p}_2$ belong to the fiber $\Pi^{-1}(p)$ for some $p\in M.$ Then there exists
a deck transformation $\sigma$ such that $\sigma(\tilde{p}_1)=\tilde{p}_2.$ Using
Lemma \ref{deckbig} and \eqref{phitheta}, we obtain 
\begin{eqnarray} 
\tilde{g}_{\bold{a},{\pmb{\theta}}}(\tilde{p}_2)
&=&\tilde{g}_{\bold{a},{\pmb{\theta}}}\circ\sigma(\tilde{p}_1) \nonumber\\
&=&\Phi_{\pmb{\theta}}(\sigma)\circ\tilde{g}_{\bold{a},{\pmb{\theta}}}(\tilde{p}
_1) \nonumber\\
&=&\tilde{g}_{\bold{a},{\pmb{\theta}}}(\tilde{p}_1).\nonumber
\end{eqnarray} 
Then $\tilde{g}_
{\bold{a},{\pmb{\theta}}}$ factors as $F\circ\Pi,$ where $F\colon M\to\mathbb{S}
^{n}$ is a minimal surface. It remains to prove that $F\in\mathcal{M}^{\mathrm{e}}_n(g).$ Since $\Pi$ is an
orientation preserving local isometry, it is obvious that $F$ is an exceptional surface.
\end{proof}

The following theorem provides properties of exceptional surfaces that are locally isometric to a pseudoholomorphic curve in
$\mathbb{S}^5.$
\begin{theorem}\label{lineorfinite}
If $g$ is a nonflat pseudoholomorphic curve in $\mathbb{S}^5,$ and $n$ is odd, then
 the moduli space $\mathcal{M}_n^{\mathrm{e}}(g)$ splits as $\mathbb{S}^{m-1}_*\times\Gamma_0,$ where $\Gamma_0$
is a subset of $\Gamma^m.$ If $\Gamma_0$ is a proper subset of $\Gamma^m,$ then
it is locally a disjoint finite
union of $d$-dimensional real analytic subvarieties where $d=0,\dots,m-1$. Moreover,
the subset $\Gamma_0$ has the property that for each point
${\pmb{\theta}}\in\Gamma_0,$ every straight line through ${\pmb{\theta}}$ that
is parallel to every coordinate axis of $\mathbb{R}^m$ either intersects $\Gamma_0$ at finitely
many points, or at a line segment.
\end{theorem}
\begin{proof}
Lemma \ref{lemmaphitheta} implies that $\mathbb{S}_*^{m-1}\times
\left\{{\pmb{\theta}}\right\}$ is contained in $\mathcal{M}_n^{\mathrm{e}}(g)$ for
each $(\bold{a},{\pmb{\theta}})\in\mathcal{M}_n^{\mathrm{e}}(g).$ Therefore,
the moduli space splits as
$$
\mathcal{M}_n^{\mathrm{e}}(g)=\mathbb{S}_*^{m-1}\times\Gamma_0,
$$
where $\Gamma_0$ is a subset of $\Gamma^m.$ Additionally, Lemma \ref{lemmaphitheta}
implies that ${\pmb{\theta}}\in\Gamma_0$ if and only if $\Phi_{\pmb{\theta}}(\mathcal{D})
=\left\{\mathrm{Id}\right\}.$ Fix $\sigma\in\mathcal{D}$. Then $\Phi_{\pmb{\theta}}(\sigma)
=\mathrm{Id}$ and $\Gamma_0$ is a real analytic set (see Remark \ref{remark1}). If $\Gamma_0$ is a proper subset of
$\Gamma^m$, according to Lojasiewicz's structure theorem \cite[Theorem
6.3.3]{KP}) the set $\Gamma_0$ locally decomposes as
\[
\Gamma_0=\mathcal{V}^0\cup\mathcal{V}^1\cup\cdots\cup\mathcal{V}^{m-1},
\]
where each $\mathcal{V}^d,\ 0\le d\le m-1,$ is either empty or a disjoint finite union
of $d$-dimensional real analytic subvarieties.

Let ${\pmb{\theta}}=(\theta_1,\dots,\theta_l,\dots,\theta_m)\in\Gamma_0.$ Suppose that the
straight line through ${\pmb{\theta}}$ that is parallel to the $l$-th coordinate axis of $\mathbb{R}^m$
is not a finite set. Thus, this line contains a sequence ${\pmb{\theta}}^{(i)}=(\theta_1,\dots,
\theta_l^{(i)},\dots,\theta_m), i\in\mathbb{N}.$ By passing if necessary to a subsequence, we
may assume that this sequence converges to ${\pmb{\theta}}^{\infty}=(\theta_1,\dots,
\theta_l^{\infty},\dots,\theta_m),$ where $\theta_l^{\infty}=\lim\theta_l^{(i)}.$
Clearly ${\theta}_{l-1}\le{\theta}_l^{\infty}\le{\theta}_{l+1}.$
At first we suppose that ${\theta}_{l-1}<{\theta}_l^{\infty}<{\theta}_{l+1},$ that is ${\pmb{\theta}}^{\infty}\in\Gamma_0.$ Fix $\sigma\in
\mathcal{D}.$ Lemma \ref{lemmaphitheta} implies that
$\Phi_{{\pmb{\theta}}^{(i)}}(\sigma)=\mathrm{Id}$ and consequently $\Phi_{{\pmb{\theta}}^{\infty}}(\sigma)=\mathrm{Id}$. 
We define the function
$$
h(\theta)=\left(\Phi_{({\theta}_1,\dots,
{\theta}_{l-1},\theta,{\theta}_{l+1},\dots{\theta}_m)}(\sigma)\right)_{ij}, \, \theta\in[\theta_{l-1},\theta_{l+1}),
$$
where $\big(\Phi_{\pmb{
\theta}}(\sigma)\big)_{ij}$ denotes the $(i,j)$-element of the matrix of $\Phi_{\pmb{
\theta}}(\sigma)$ with respect to the standard basis of $\mathbb{R}^{n+1}$.
From the mean value theorem we have that there exists $\xi_1^{(i)}$ between $\theta_l^{(i)}$ and
${\theta}_l^{\infty}$ such that $(dh/d\theta)(\xi_1^{(i)})=0$ and hence $(dh/d\theta)(\theta_l^{\infty})
=0.$ Applying
again the mean value theorem, we obtain
that there exists $\xi_2^{(i)}$ between $\xi_1^{(i)}$ and $\theta_l^{\infty}$ such that
$(d^2h/d\theta^2)(\xi_2^{(i)})=0.$ Inductively, we have
that the $k$-th derivative satisfies $(d^kh/d\theta^k)(\theta_l^{\infty})
=0$ for any 
$k.$ The analyticity of $h$ (see Remark \ref{remark1})
yields that $h=\delta_{ij}$ on $[\theta_{l-1},\theta_{l+1}),$ where $\delta_{ij}$ is the Kr\"onecker delta.

Now without loss of generality, assume that ${\theta}_{l-1}={\theta}_l^{\infty}<{\theta}_{l+1}.$
Clearly ${\pmb{\theta}}^{\infty}\notin\Gamma_0.$ We fix $\sigma\in\mathcal{D}$
and extend $\Phi_{\pmb{\theta}}$ in the obvious way. Then $\Phi_{{\pmb{\theta}}^{(i)}}
(\sigma)=\mathrm{Id}$ and consequently $\Phi_{{\pmb{\theta}}^{\infty}}(\sigma)=\mathrm{Id}$
and the claim follows as before.
\end{proof}

We now provide a result for compact pseudoholomorphic curves in $\mathbb{S}^5.$

\begin{theorem}\label{compactps}
If $g$ is a compact pseudoholomorphic curve in $\mathbb{S}^5$ that is not
homeomorphic to the torus, then the moduli space $\mathcal{M}_n^{\mathrm{e}}
(g),$ with $n$ odd, is given by $\mathcal{M}_n^{\mathrm{e}}(g)=\mathbb{S}^{m-1}_*\times
\Gamma_0,$ where $\Gamma_0$ is a proper subset  of $\Gamma^m$ that is locally a disjoint finite
union of $d$-dimensional real analytic subvarieties where $d=0,\dots,m-1.$ Moreover,
every straight line through each point ${\pmb{\theta}}\in\Gamma_0$ that
is parallel to every coordinate axis of $\mathbb{R}^m$ intersects $\Gamma_0$ at finitely
many points.
\end{theorem}
\begin{proof}
Suppose to the contrary that the intersection of $\Gamma_0$ with the straight line through $\pmb
{\theta}$ that is parallel to the first coordinate axis is an infinite set. For a fixed $\bold{a}\in
\mathbb{S}_*^{m-1},$ we choose $\pmb{\theta}_1,\dots, \pmb
{\theta}_N\in\Gamma_0$ that belong to this straight line. Hence $(\bold{a},\pmb
{\theta}_j)\in\mathcal{M}^{\mathrm{e}}_n(g)$ for all $\pmb{\theta}_j=(
\theta_{j1},\dots,\theta_{jm}),\ j=1,\dots,N.$ Consequently there exist exceptional
surfaces $F_j\colon M\to\mathbb{S}^n$ such that $F_j\circ\pi=\tilde{g}_{\bold{a},
\pmb{\theta}_j}.$

We claim that the set of all coordinate functions associated to vectors $\bold{v}=(v_1,0,\dots,0)$ in $\mathbb{R}^{6m}$
of all surfaces $F_j$'s are linearly independent. It is
sufficient to prove that if
\begin{equation}\label{linear}
\sum\limits_{j=1}^N\langle F_j,\bold{v}\rangle=0,
\end{equation}
then $\bold{v}=0.$ From \eqref{linear} we obtain 
$$
\sum\limits_{j=1}^N\langle F_j\circ\pi,\bold{v}\rangle=0,
$$
or equivalently 
$$
a_1\sum\limits_{j=1}^N\langle \tilde{g}_{\theta_{j1}},v_1\rangle=0.
$$
In analogy with the argument in the proof of \cite[Theorem 2]{DV}, we finally conclude
that $v_1=0$ and the claim is proved.

The contradiction follows easily since the coordinate functions of the surfaces $F_j$'s are eigenfunctions of
the Laplacian operator with corresponding eigenvalue 2 and the vector space of the eigenfunctions 
has finite dimension. Hence $\Gamma_0\neq\Gamma^m$ and the proof follows from Theorem \ref{lineorfinite}.
\end{proof}

\begin{remark}\label{flattori}
The assumption in Theorem \ref{compactps} that the pseudoholomorphic curve
$g$ is not homeomorphic to the torus
is essential and can not be dropped. According to results due to Kenmotsu \cite{K2, K11}
the moduli space of all minimal surfaces in odd dimensional spheres that are isometric to
a flat pseudoholomorphic torus in $\mathbb{S}^5$ is not a finite set. 
\end{remark}

\begin{proof}[Proof of Theorem \ref{mikradim}]
It follows from \cite[Theorem 5]{VT} and \cite[Corollary 1]{VT} that any minimal surface $f\colon M\to\mathbb{S}^n$ that is
isometric to $g$ is exceptional and $n=5.$ Then Theorem \ref{compactps} above completes the proof.
\end{proof}

\section{Rigidity of isotropic pseudoholomorphic curves in $\mathbb{S}^6$}
In this section, we study the rigidity of isotropic pseudoholomorphic curves in $\mathbb{S}
^6$ among minimal surfaces in spheres. We prove Theorem \ref{forintro}.

\begin{proof}[Proof of Theorem \ref{forintro}]
According to \cite[Theorem 2]{V16}, the function $1-K$ is of absolute value type.
If the zero set of the function $1-K$ is empty, then from condition ($\ast\ast$) and the Gauss-Bonnet theorem
it follows that $M$ is homeomorphic to the sphere. From \cite{Calabi} we have that $f$ is isotropic and from
\cite{V} it follows that $n=6$ and $f$ is congruent to $g.$
Now suppose that the zero set of the function $1-K$
is the finite set $M_0=\left\{ p_1,\dots,p_m\right\}$ with corresponding order $\mathrm{ord}_{p_j}(1-K)=2k_j.$
For each point $p_j, j=1,\dots,m,$ we choose a
local complex coordinate $z$ such that $p_j$ corresponds to $z=0$\ and the induced
metric is written as $ds^2=F|dz|^2.$ On a neighbourhood of \ $p_j,$ we have that
\begin{equation}\label{againavt}
1-K=|z|^{2k_j}u_0,
\end{equation}
where $u_0$ is a smooth positive function.

We claim that $f$ is $1$-isotropic. The first
Hopf differential $\Phi _1=f_1dz^4$ is globally defined and holomorphic. Hence either $\Phi _1$ is identically zero, or its zeros are isolated.  Suppose to the contrary
that $\Phi _1$ is not identically zero. The Gauss equation \eqref{Gausseqa2} yields that each $p_j$ is a totally geodesic point. It follows from the definition of the first Hopf differential that $\Phi _1$ vanishes at each $p_j.$
Hence we may write $f_1=z^{l_{1j}}\psi _1$\ around\ $p_j,$ where
$l_{1j}$\ is the order of $\Phi_1$\ at $p_j,$ and $\psi _1$ is a nonzero
holomorphic function. Bearing in mind \eqref{what}, we obtain
\begin{equation}\label{forbelow}
\frac{1}{4}\left\Vert \alpha_2\right\Vert^4-(K_1^{\perp})^2=(2F^{-1})^{4}|\psi _1|^2|z|^{2l_{1j}}
\end{equation}
around $p_j.$ We now consider the function $u_1\colon M
\smallsetminus M_0\to \mathbb{R}$ defined by
$$
u_1=\frac{\left(\frac{1}{4}\left\Vert \alpha_2\right\Vert^4-(K_1^{\perp})^2\right)^3}{(1-K)^4}.
$$
From (\ref{againavt}) and \eqref{forbelow} it follows that the function $u_1$
around $p_j,$ is written as 
\begin{equation}\label{u_1around}
u_1=(2F^{-1})^{12}u_0^{-4}|\psi _1|^6|z|^{6l_{1j}-8k_j}.
\end{equation}
Using \eqref{Gausseqa2} we obtain $u_1\leq (1-K)^2.$ Thus, from \eqref{againavt} and \eqref{u_1around} we deduce that $l_{1j}\geq 2k_j$ and we
can extend $u_1$ to a smooth function on $M.$ From Proposition \ref{3i}(ii) for $r=1,$ it follows that
\begin{equation*}
\Delta \log \left(\left\Vert \alpha_{2}\right\Vert^2+2K_1^{\perp}\right) =2\big(2K-K_1^{\ast}\big)
\end{equation*}
and 
\begin{equation*}
\Delta \log \left(\left\Vert \alpha_{2}\right\Vert^2-2K_1^{\perp}\right) =2\big(2K+K_1^{\ast}\big).
\end{equation*}
Summing up, we obtain
\begin{equation*}
\Delta \log \left(\left\Vert \alpha_{2}\right\Vert^4-4(K_1^{\perp})^2\right) =8K.
\end{equation*}
Combining the last equation with the condition ($\ast\ast$), we have
\begin{equation*}
\Delta \log \left(\left\Vert \alpha_{2}\right\Vert^4-4(K_1^{\perp})^2\right)^3=\Delta\log (1-K)^4+4,
\end{equation*}
or equivalently $\Delta\log u_1=4$ away from the isolated
zeros of $u_1$. Thus, by continuity $\Delta u_1\ge 4u_1\ge0,$ and from the maximum principle we have that this holds only for $u_1\equiv0,$ or equivalently only if $\left\Vert \alpha_2\right\Vert^4=4(K_1^{\perp})^2.$ Lemma \ref{systematic} implies that $\Phi _1=0$ and this contradicts our assumption that $\Phi _1$ is not identically zero. Hence, $\Phi _1$ is identically zero and from Lemma \ref{systematic} yields that $f$ is 1-isotropic.
Proposition \ref{3i}(i) for $s=1$ implies that
\begin{equation*}
\Delta \log (1-K)=2\big(2K-K_1^{\ast}\big),
\end{equation*}
which using the condition ($\ast\ast$) yields
\begin{equation}\label{k*fin}
K_1^{\ast}=\frac{1}{2}-K.
\end{equation}
Since $f$ is 1-isotropic, we know from Lemma \ref{systematic} that $K_1^\perp=1-K.$ Proposition \ref{5} and \eqref{k*fin} yield that 
\begin{equation}\label{a3sketo}
\left\Vert \alpha_{3}\right\Vert^2=1-K.
\end{equation}

We now claim that $f$ is also $2$-isotropic. From Proposition \ref{neoksanaafththfora} we know that $\Phi _2=f_2dz^{6}$
is globally defined. Theorem \ref{ena} implies that it is also holomorphic. Hence either $\Phi _2$ is identically zero or its zeros are isolated.
In the former case, from Lemma \ref{systematic} we have that $f$ is $2$-isotropic. Assume now to the contrary that $\Phi _2$ is not identically zero. Obviously, we have that $\alpha_3$ vanishes at each $p_j$ and consequently from the definition of the second Hopf differential, also $\Phi _2$ vanishes at each
$p_j.$ Hence we may write $f_2=z^{l_{2j}}\psi _2$\ around\ $p_j,$
where $l_{2j}$\ is the order of $\Phi _2$\ at $p_j,$ and $\psi _2$ is a
nonzero holomorphic function. Bearing in mind (\ref{what}),\ we obtain 
\begin{equation}\label{eqa3}
\left\Vert \alpha_{3}\right\Vert^4-16(K_2^{\perp})^2=2^{8}F^{-6}|\psi _2|^2|z|^{2l_{2j}}
\end{equation}
around $p_j.$  We now consider the function $u_2\colon M
\smallsetminus M_0\to \mathbb{R}$ defined by
$$
u_2=\frac{\left\Vert \alpha_{3}\right\Vert^4-16(K_2^{\perp})^2}{(1-K)^{2}}.
$$
 In view of (\ref{againavt}) and \eqref{eqa3}, it follows that the function $u_2$ around $p_j$ is written as
\begin{equation}\label{u2}
u_2=2^{8}F^{-6}u_0^{-2}|\psi
_2|^2|z|^{2l_{2j}-4k_j}.
\end{equation}

Using \eqref{a3sketo} we derive that $u_2\leq 1.$ From \eqref{againavt} and (\ref{u2}) we deduce that $l_{2j}\geq 2k_j$ and we
can extend $u_2$ to a smooth function on $M.$ Proposition \ref{3i}(ii) for $r=2$ implies that 
\begin{equation*}
\Delta \log \left(\left\Vert \alpha_{3}\right\Vert^2+4K_2^{\perp}\right) =2\big(3K-K_2^{\ast}\big)
\end{equation*}
and 
\begin{equation*}
\Delta \log \left(\left\Vert \alpha_{3}\right\Vert^2-4K_2^{\perp}\right) =2\big(3K+K_2^{\ast}\big).
\end{equation*}
Summing up, we obtain
\begin{equation*}
\Delta \log \left(\left\Vert \alpha_{3}\right\Vert^4-16(K_2^{\perp})^2\right) =12K.
\end{equation*}
Combining the last equation with the condition ($\ast\ast$), we have that $\Delta\log u_2=2$ away from the isolated
zeros of $u_2$. Thus, by continuity $\Delta u_2\ge 2u_2\ge0,$ and from the maximum principle we have that this holds only for $u_2\equiv0,$  or equivalently only if $\left\Vert \alpha_3\right\Vert^4=16(K_2^{\perp})^2.$ Lemma \ref{systematic} implies that $\Phi _2=0$ and this contradicts our assumption that $\Phi _2$ is not identically zero. Hence, $\Phi _2$ is identically zero and Lemma \ref{systematic} implies that $f$ is 2-isotropic.
Now Proposition \ref{3i}(i) for $s=2$ yields
\begin{equation*}
\Delta \log \left\Vert \alpha_{3}\right\Vert^2 =2\big(3K-K_2^{\ast}\big),
\end{equation*}
 and combining this with condition $(\ast\ast)$ we obtain $K_2^{\ast}=1/2.$

Since $f$ is 2-isotropic, from Lemma \ref{systematic} and \eqref{a3sketo} we have $K_2^\perp=(1-K)/4.$ Using that $K_2^{\ast}=1/2,$ 
Proposition \ref{5} for $r=2$ implies that $\alpha_{4}=0.$ Therefore $n=6$ and the surface $f$ is congruent to $g$ (cf. \cite[Theorem A]{V}).
\end{proof}

We now prove the following local result for exceptional surfaces.
\begin{theorem} \label{congruent2iso}
Let $f\colon M\to\mathbb{S}^n$ be a substantial exceptional surface that is isometric to an isotropic pseudoholomorphic curve $g\colon M \to\mathbb{S}^6.$ Then $n=6$ and $f$ is congruent to $g.$
\end{theorem}
\begin{proof}
We set $\rho _s:=2^sK_s^{\perp}/\left\Vert \alpha_{s+1}\right\Vert^2,$ for any $1\leq s\leq r,$ where $r=[(n-1)/2-1].$ Using \eqref{elipsi} and \eqref{si} it follows that $\rho _s=2\kappa _{s}\mu _{s}/(\kappa _{s}^2+\mu _{s}^2).$ 
Since $f$ is exceptional, by the definition we have that the $s$-th ellipse has constant eccentricity or equivalently the ratio of the semiaxes $\kappa _{s},\mu _{s}$ is constant. Then it is clear that the function $\rho _s$ is constant.

Using equation \eqref{Gausseqa2}, Proposition \ref{3i}(i) for $s=1$ and condition  ($\ast\ast$), we find 
\begin{equation}\label{K*now}
K_1^{\ast}=\frac{1}{2}-K.
\end{equation}
Moreover, from the definition of $\rho_1$ we have that $K_1^{\perp}=\rho_1(1-K).$ 
We claim that $f$ is 1-isotropic, which is equivalent to $\rho_1=1$ due to Lemma \ref{systematic}. Assume to the contrary that
$\rho_1\neq1.$ Then from Lemma \ref{systematic} we have that $\Phi_1\neq0.$ Consequently Proposition \ref{3i}(ii) for $r=1$ yields 
\begin{equation*}
\Delta \log \left(\left\Vert \alpha_{2}\right\Vert^2-2K_1^{\perp}\right) =2\big(2K+K_1^{\ast}\big).
\end{equation*}
Using \eqref{Gausseqa2} and Lemma \ref{systematic}(iii) we obtain
\[
\Delta\log(1-K)=4K+2K_1^{\ast}.
\]
From \eqref{K*now} it follows that 
\[
\Delta\log(1-K)=2K+1.
\]
Combining this with
condition  ($\ast\ast$) we have that $K=1/2$, which is a contradiction. Hence
$\rho_1=1$ and consequently $f$ is 1-isotropic. From Proposition \ref{5}, equation \eqref{K*now} and Lemma \ref{systematic} it follows that 
\begin{equation}\label{a3laast}
\left\Vert \alpha_{3}\right\Vert^2=1-K.
\end{equation}

From Proposition \ref{neoksanaafththfora} we know that $\Phi _2=f_2dz^{6}$
is globally defined. Theorem \ref{ena} implies that it is also holomorphic. Hence either $\Phi _2$ is identically zero or its zeros are isolated.
Moreover, we have that $K_2^\perp=2^{-2}\rho_2\left\Vert \alpha_{3}\right\Vert^2.$ Similarly, we claim that $\rho _{2}=1.$ 
Assume to the contrary that
$\rho_2\neq1.$ Then from Lemma \ref{systematic}, the Hopf differential $\Phi_2\neq0.$ Proposition \ref{3i}(ii) for $r=2$ yields that
\begin{equation*}
\Delta \log \left(\left\Vert \alpha_{3}\right\Vert^2+4K_2^{\perp}\right) =2\big(3K-K_2^{\ast}\big)
\end{equation*}
and 
\begin{equation*}
\Delta \log \left(\left\Vert \alpha_{3}\right\Vert^2-4K_2^{\perp}\right) =2\big(3K+K_2^{\ast}\big),
\end{equation*} 
which due to \eqref{a3laast} implies that
\begin{equation*}
\Delta\log(1-K)=6K.
\end{equation*}
This contradicts $(\ast\ast),$ hence $\rho_2=1$ and consequently $\Phi_2$ is identically zero.
From Proposition \ref{3i}(iii) for $r=2$ and condition  ($\ast\ast$), we obtain $K_{2}^
{\ast}=1/2.$ Proposition \ref{5} for $r=2$ yields 
$ \alpha_{4}=0,$ which completes our proof.
\end{proof}

\section{Isometric deformations of nonisotropic pseudoholomorphic curves in $\mathbb{S}^6$}
In this section, we mostly deal with noncongruent isometric deformations of
pseudoholomorphic curves in $\mathbb{S}^6$ that are always 1-isotropic (see \cite{V16}) but
in general not 2-isotropic.

For a given nonisotropic pseudoholomorphic curve
$g\colon M\to\mathbb{S}^6,$ 
our aim is to describe the moduli space $\mathcal{M}
^K_n(g)$
of all noncongruent minimal surfaces $f\colon M\to\mathbb{S}^n$ that are locally
isometric to the curve $g$, having the same normal curvatures up to order 2 with
the curve $g.$

From \cite[Corollary 5.4(ii)]{V} we know that two
locally isometric 1-isotropic surfaces in $\mathbb{S}^6$ with the same normal curvatures, belong locally
to the same associated family. In particular, if $g$ is simply connected then $\mathcal{M}^K_6
(g)=[0,\pi).$

Hereafter we are interested in the case where the pseudoholomorphic curve $g$ is nonsimply connected. We consider the covering map $\Pi\colon\tilde{M}\to M,$ $\tilde{M}$ being the universal cover of $M$ equipped with the metric and
orientation that make $\Pi$ an orientation preserving local isometry. Corresponding objects on
$\tilde{M}$ are denoted with tilde. Then the map $\tilde{g}\colon \tilde{M}\to\mathbb{S}^6$
with $\tilde{g}=g\circ\Pi$ is up to congruence a pseudoholomorphic curve. Hence, the moduli space $\mathcal{M}^K_6(g)$ of
the curve $g$ can be described as the set of all $\theta\in\mathcal{M}^K_6
(\tilde{g})=[0,\pi)$ such that $\tilde{g}_{\theta}$ factors as $\tilde{g}_{\theta}=g_\theta\circ\Pi$
for a minimal surface $g_\theta\colon M\to\mathbb{S}^6$
and $\tilde{g}_{\theta}$ is a member in the associated family of $\tilde{g}.$
We follow this notation throughout this section.

\begin{lemma}\label{lemmaPhitheta}
(i) For each $\sigma\in\mathcal{D},$ the surfaces $\tilde{g}_{\theta}$
and $\tilde{g}_{\theta}\circ\sigma$ are congruent for every $\theta\in[0,\pi],$ that is there exists $\Psi_{\theta}(\sigma)\in\mathrm{O}
(7)$ such that
\begin{equation} \label{deckagain}
\tilde{g}_
{\theta}\circ\sigma=\Psi_{\theta}(\sigma)\circ\tilde{g}_
{\theta}.
\end{equation}

(ii) If $\theta$ belongs to $\mathcal{M}^K_6(\tilde{g}),$ then $\theta$ belongs to $\mathcal{M}^K_6(g)$ if
and only if
\begin{equation}\label{phithetaIII}
\Psi_{\theta}(\mathcal{D})=\left\{\mathrm{Id}\right\},
\end{equation} 
where $\Psi_{\theta}\in\mathrm{O}(7).$
\end{lemma}
\begin{proof}
(i) From \cite[Proposition 9]{DV} we have that for any $\sigma$ in the group
$\mathcal{D},$ the surfaces $\tilde{g}_\theta\colon\tilde{M}\to\mathbb{S}^6$ and $\tilde{g}_\theta\circ\sigma\colon \tilde{M}\to\mathbb{S}^6$ are congruent for any $\theta\in\mathcal{M}^K_6(g).$
Therefore, there exists 
$\Psi_\theta(\sigma)\in\mathrm{O}(7)$ such that \eqref{deckagain} holds
for every $\theta\in\mathcal{M}^K_6(g).$

(ii) Take $\theta\in\mathcal{M}^K_6(g).$ Then, there exists a minimal surface $g_{\theta}\colon M\to\mathbb
{S}^{6}$ such that $g_{\theta}\circ\pi=\tilde{g}_{\theta}.$
Composing with an arbitrary $\sigma\in\mathcal{D}$ and using
\eqref{deckagain} we obtain
\begin{equation*}
\tilde{g}_{\theta}=\Psi_{\theta}(\sigma)\circ\tilde{g}_{\theta}.
\end{equation*}
Since $\tilde{g}_{\theta}$ has substantial
codimension \eqref{phithetaIII} yields.

Conversely assume that \eqref{phithetaIII} holds. We will prove that 
$\tilde{g}_{\theta}$ factors as $\tilde{g}_{\theta}=g_{\theta}\circ\Pi,$ where $g_{\theta}\colon M\to\mathbb{S}^6$ is
a minimal surface. At first we claim that $\tilde{g}_{\theta}$ remains constant on each fiber of
the covering map $\Pi.$ Indeed, let
$\tilde{p}_1, \tilde{p}_2$ belong to $\Pi^{-1}(p)$ for some $p\in M.$ Then there exists
a deck transformation $\sigma$ such that $\sigma(\tilde{p}_1)=\tilde{p}_2.$ Using
\eqref{deckagain}, we obtain 
\begin{eqnarray} 
\tilde{g}_{\theta}(\tilde{p}_2)
&=&\tilde{g}_{\theta}\circ\sigma(\tilde{p}_1) \nonumber\\
&=&\Psi_{\theta}(\sigma)\circ\tilde{g}_{\theta}(\tilde{p}
_1) \nonumber\\
&=&\tilde{g}_{\theta}(\tilde{p}_1).\nonumber
\end{eqnarray} 
Then $\tilde{g}_{\theta}$ factors as $\tilde{g}_{\theta}=g_{\theta}\circ\Pi,$ where $F\colon M\to\mathbb{S}
^{n}$ is a minimal surface. It remains to prove that $g_{\theta}\in\mathcal{M}^K_6(g).$ Since $\Pi$ is an
orientation preserving local isometry, it is obvious that $F$ is a minimal surface.
\end{proof}
Now we are able to prove Theorem \ref{finiteorcircle}.
\begin{proof}[Proof of Theorem \ref{finiteorcircle}]
If $g$ is substantial in a totally geodesic $\mathbb{S}^5,$ then from \cite[Theorem 1]{DV},
the moduli space of $g$ is either a circle or a finite set.

If $g$ is isotropic and substantial in $\mathbb{S}^6,$ then Theorem \ref{congruent2iso}
implies that the moduli space of $g$ consists of a single point.

Suppose now that $g$ is substantial in $\mathbb{S}^6$ and nonisotropic. Assume that $\mathcal{M}^K_6(g)$ is not finite. Thus, there exists a sequence $\theta^{(i)}, i\in
\mathbb{N},$ that belongs to $\mathcal{M}^K_6(g).$ By passing if necessary to a subsequence, we
assume that this sequence converges to $\theta^{\infty}\in[0,\pi].$
From Lemma \ref{lemmaPhitheta}(ii), we derive that $\Psi_{\theta^{(i)}}(\mathcal{D})=\left\{\mathrm{Id}\right\}$ for
every $i\in\mathbb{N}$ and $\Psi_{\theta^{\infty}}(\mathcal{D})=\left\{\mathrm{Id}\right\}.$
Fix a $\sigma\in\mathcal{D}.$ 
We define the function
$$
h(\theta)=\left(\Psi_{\theta}(\sigma)\right)_{ij}, \theta\in[0,\pi],
$$
where $\left(\Psi_{\theta}(\sigma)\right)_{ij}$ denotes the $(i,j)$-element of the matrix of $\Psi_
{\theta}(\sigma)$ with respect to the standard basis of $\mathbb{R}^{7}$.
By the mean value theorem, there exists $\xi_1^{(i)}$ between $\theta^{(i)}$ and
${\theta}^{\infty}$ such that $(dh/d\theta)(\xi_1^{(i)})=0$ and hence $(dh/d\theta)(\theta^{\infty})
=0.$ Applying
repeatedly the mean value theorem, we obtain inductively
that the $k$-th derrivative satisfies $(d^kh/d\theta^k)(\theta^{\infty})
=0$ for any $k.$ The analyticity of $h$ (cf. \cite{EQ}) implies that $h=\delta_{ij},$ where $\delta_{ij}$ is the Kr\"onecker delta.
\end{proof}
We now turn our attention to the study of isometric deformations of compact nonisotropic pseudoholomorphic curves in $\mathbb{S}^6.$ We will need the following lemmas:

\begin{lemma}\label{antistoixo}
Let $g\colon M\to\mathbb{S}^6$ be a nonisotropic pseudoholomorphic curve.
For each $g_\theta,\theta\in\mathcal{M}^K_6(g),$ there exists a parallel and orthogonal
bundle isomorphism $T_\theta\colon N_gM\to N_{g_\theta}M$ such that the second
fundamental forms of $g$ and $g_\theta$ are related by
\[
\alpha^{g_\theta}(X,Y)=T_\theta\circ\alpha^g(J_\theta X,Y),\ \ X,Y\in TM.
\]
\end{lemma}
\begin{proof}
Since $g$ and $g_\theta$ have the same normal curvatures, it follows from \cite[Corollary 5.4(ii)]{V} that for any simply connected
subset $U$ of $M$ there exists a parallel and orthogonal bundle isomorphism $T_\theta^U\colon
N_g U\to N_{g_\theta}U$ such that the second fundamental forms of the surfaces $g|_U$ and
$g_{\theta}|_U$ are related by
\[
\alpha^{g_\theta|_U}(X,Y)=T_\theta^U\circ\alpha^{g|_U}(J_\theta X,Y),\ \ X,Y\in TM.
\]
Let $U,V$ be simply connected subsets of $M$ with $U\cap V\neq\varnothing.$
Then on $U\cap V$ we have
\[
T_\theta^U\circ\alpha^{g|_U}(J_\theta X,Y)=T_\theta^V\circ\alpha^{g|_V}(J_\theta X,Y),
\]
for every $X,Y\in TM.$ Equivalently we obtain 
\[
\left(T_\theta^U-T_\theta^V\right)\circ\alpha^{g|_{U\cap V}}(X,Y)=0
\]
and obviously $\left(T_\theta^U-T_\theta^V\right)(N^{g|_{U\cap V}}_1)=0.$

Differentiating we obtain $\left(T_\theta^U-T_\theta^V\right)(N^{g|_{U\cap V}}_2)=0,$ which yields
that $T_\theta^U=T_\theta^V$ on $U\cap V.$ Thus,
$T_\theta^U$ is globally well defined.
\end{proof}
For each orthonormal frame along any minimal surface, one has the connection forms (cf. \cite{V}).
\begin{lemma}\label{ksanauseful}
Let $g\colon M\to\mathbb{S}^6$ be a substantial nonisotropic pseudoholomorphic curve and
let $M_1$ be the zero set of the second Hopf differential $\Phi_2.$ Around each point of $M\smallsetminus M_1,$ there exist a local
complex coordinate $(U,z),$ $U\subset M\smallsetminus M_1$ and orthonormal
frames $\{e_1,e_2\}$ in $TU,$ $\{e_3,e_4\}$ in $N^g_1U$ and $\{e_5,e_6\}$ in $N^g_2U$ which agree with the given orientations such that:


(i) $e_5$ and $e_6$ give respectively the directions of the major and the minor axes of the second curvature ellipse, and

(ii) $H_5=\kappa_2,$ $H_6=i\mu_2$ and $\kappa_2$ and $\mu_2$ are smooth real functions. Moreover, the connection and the normal connection forms, with respect to this frame, are given respectively, by 
\begin{equation}\label{connectionforms}
\omega_{12}=-\frac{1}{6}*d\log(\kappa_2^2-\mu_2^2),\,\,\omega_{34}=2\omega_{12}+*d\log\kappa_1, \,\, \omega_{56}=\frac{\kappa_2\mu_2}{\kappa_2^2-\mu_2^2}*d\log\frac{\mu_2}{\kappa_2}, 
\end{equation}
where $*$ stands for the Hodge operator.
\end{lemma}
\begin{proof}
(i) Take an arbitrary orthonormal frame $\{E_1,E_2\}$ in $TU.$ Arguing
pointwise in $U$ we have that
$$\max\limits_{\theta\in[0,2\pi)}\Vert\alpha^g_3(X_\theta,X_\theta,X_\theta)\Vert=
\kappa_2 \text{\,\, and\,\,}
\min\limits_{\theta\in[0,2\pi)}\Vert\alpha^g_3(X_\theta,X_\theta,X_\theta)\Vert=\mu_2,
$$
where $X_\theta=\cos\theta E_1+\sin\theta E_2.$
Assume that the function $f(\theta)=\Vert\alpha^g_3(X_\theta,X_\theta,X_\theta)\Vert^2$ attains
its maximum at $\theta_0\in[0,2\pi).$ Since $f'(\theta_0)=0,$ we find that 
\[
\langle \alpha^g_3(X_{\theta_0},X_{\theta_0},X_{\theta_0}),\alpha^g_3(X_{\theta_0},X_{\theta_0},X_{\theta_0})\rangle=0,
\]
or equivalently
\[
2\langle \alpha^g_3(E_1,E_1,E_1),\alpha^g_3(E_1,E_1,E_2) \rangle\cos 6\theta =
\left( \Vert\alpha^g_3(E_1,E_1,E_1)\Vert^2-\Vert\alpha^g_3(E_1,E_1,E_2)\Vert^2\right)\sin 6\theta.
\]
Since the second curvature ellipse is not a circle, we choose a smooth function 
$\sigma$ such that
\[
\tan\sigma=
\frac{2\langle \alpha^g_3(E_1,E_1,E_1),\alpha^g_3(E_1,E_1,E_2) \rangle}
{\Vert\alpha^g_3(E_1,E_1,E_1)\Vert^2-\Vert\alpha^g_3(E_1,E_1,E_2)\Vert^2},
\]
or
\[
\cot\sigma=
\frac{\Vert\alpha^g_3(E_1,E_1,E_1)\Vert^2-\Vert\alpha^g_3(E_1,E_1,E_2)\Vert^2}
{2\langle \alpha^g_3(E_1,E_1,E_1),\alpha^g_3(E_1,E_1,E_2) \rangle}.
\]

We now consider the orthonormal frame $\{e_1,e_2\}$ in $TU$ with
$$ e_1=\cos\sigma E_1+\sin\sigma
E_2 \text{\,\,and\,\,} e_2=-\sin\sigma E_1+\cos\sigma E_2.$$
We may also consider the orthonormal frame $\{e_3, e_4\}$
in $N^g_1U$ given by 
$$ e_3=\frac{1}{\kappa_1 }\alpha^g(e_1,e_1)\text{\,\,and\,\,} 
e_4=\frac{1}{\kappa_1}\alpha^g(e_1,e_2)$$
and the orthonormal frame $\{e_5, e_6\}$ in $N^g_2U$ such
that
$$ e_5= \frac{1}{\kappa_2}\alpha^g_3(e_1,e_1,e_1)\text{\,\,and\,\,} e_6=\frac{1}{\mu_2}\alpha^g_3(e_1,e_1,e_2).$$

Let $\{\tilde{e}_5,
\tilde{e}_6\}$ be an orthonormal frame in $N^g_2U$ as in \cite[Lemma 5]{V16}. Then 
the complex valued functions $\tilde{H}_5,\tilde{H}_6$ associated to the
frame $\{\tilde{e}_5, \tilde{e}_6\}$ satisfy
\begin{equation}\label{tildeH}
\tilde{H}_6
=i(\kappa_1-\tilde{H}_5).
\end{equation}
We easily find that

\begin{equation}\label{H5}
\tilde{H}_5=\cos\varphi H_5+\sin\varphi H_6
\end{equation}
and 
\begin{equation}\label{H6}
\tilde{H}_6=-\sin\varphi H_5+\cos\varphi H_6,
\end{equation}
where $\varphi$ is the angle between $e_5$ and $\tilde{e}_5.$
Since $H_5=\kappa_2$ and $H_6=i\mu_2,$ equations
\eqref{tildeH}, \eqref{H5} and \eqref{H6} yield $\varphi=0$ and consequently the orthonormal
frames $\{e_5,e_6\}$ and $\{\tilde{e}_5,\tilde{e}_6\}$ coincide.

(ii) It follows directly from \cite[Lemma 6]{V08} that the connection forms $\omega_{34}$ and $\omega_{56}$
are given by \eqref{connectionforms}.

From $\alpha_3(e_1,e_1,e_1)=\kappa_2e_5,$ we obtain
\[
\omega_{35}(e_1)=-\omega_{45}(e_2)=\frac{\kappa_2}{\kappa_1}
\text{\,\,and\,\,\,}
\omega_{36}(e_1)=\omega_{46}(e_2)=0.
\]
Similarly, $\alpha_3(e_1,e_1,e_2)=\mu_2e_6$ implies that
\[
\omega_{46}(e_1)=\omega_{36}(e_2)=\frac{\mu_2}{\kappa_1}
\text{\,\,and\,\,\,}
\omega_{45}(e_1)=\omega_{35}(e_2)=0.
\]
Therefore, 
\[
\omega_{35}=\frac{\kappa_2}{\kappa_1}\omega_1,\ \omega_{45}=-\frac{\kappa_2}{\kappa_1}
\omega_2,
\ \omega_{36}=\frac{\mu_2}{\kappa_1}\omega_2 \text{\,\, and\,\,}\omega_{46}=\frac{\mu_2}{\kappa_1}\omega_1.
\]
Using the above, the Ricci equations 
\[
\langle R^\perp(e_1,e_2)e_3,e_5\rangle=0 \text{\,\ and \,} \langle R^\perp
(e_1,e_2)e_4,e_6\rangle=0,
\]
where $R^\perp$ is the curvature tensor of the normal bundle, are written equivalently as
\[
3\omega_{12}(e_1)=\frac{\mu_2}{\kappa_2}\omega_{56}(e_1)+e_2\big(\log\frac
{\kappa_2}{\kappa_1}\big)-*d\log\kappa_1(e_1)
\]
and
\[
3\omega_{12}(e_1)=\frac{\kappa_2}{\mu_2}\omega_{56}(e_1)+e_2\big(\log\frac
{\mu_2}{\kappa_1}\big)-*d\log\kappa_1(e_1)
\]
respectively. From these and from the fact that the normal connection form $\omega_{56}$ is given by \eqref{connectionforms}, one can easily obtain
\begin{eqnarray}\label{omegapros}
\omega_{12}(e_1)=-\frac{1}{6}*d\log(\kappa_2^2-\mu_2^2)(e_1).
\end{eqnarray}
Arguing similarly for the Ricci equations 
\[
\langle R^\perp(e_1,e_2)e_3,e_6\rangle=0 \text{\,\ and \,} \langle R^\perp
(e_1,e_2)e_4,e_5\rangle=0
\]
we have that
\begin{eqnarray*}
\omega_{12}(e_2)=-\frac{1}{6}*d\log(\kappa_2^2-\mu_2^2)(e_2), \nonumber
\end{eqnarray*}
which combined with \eqref{omegapros} yields the connection form $\omega_{12}$ of \eqref{connectionforms}.
\end{proof}

Let $g\colon M\to\mathbb{S}^6$ be a substantial pseudoholomorphic curve. Assume hereafter
that $g$ is nonisotropic. For each point $p\in M\smallsetminus M_1,$ we consider
$\{e_1,e_2,e_3,e_4,e_5,e_6\}$ being an orthonormal frame on a neighborhood $U
\subset M\smallsetminus M_1$ of $p$ as in Lemma
\ref{ksanauseful}. 
We note that the connection form $\omega_{56}$ cannot vanish on any open subset of
$M\smallsetminus M_1.$ 
Suppose to the contrary that $\omega_{56}=0.$ Then \eqref{connectionforms} implies that
$\mu_2=\lambda \kappa_2$ for some $\lambda\in\mathbb{R}^+$ and from \cite[Theorem 5(iii)]{V16} we obtain 
\[
\kappa_2=\frac{\kappa_1}{\lambda+1} \text{\,\, and\,}\ \mu_2=\frac{\lambda\kappa_1}{\lambda+1}.
\]
From \eqref{connectionforms} it follows that the connection form is given by
\[
\omega_{12}=-\frac{1}{3}*d\log\kappa_1,
\]
which implies
\[
6K=\Delta\log{\kappa_1^2}=\Delta\log(1-K).
\]
According to Theorem \ref{eschvlach}, this would imply a reduction of codimension, which is a contradiction.

For any $\theta\in\mathcal{M}^K_6(g),$ let $\{e_1,e_2,T_\theta e_3,T_\theta e_4,T_\theta
e_5,T_\theta e_6\}$ be an orthonormal frame along $g_\theta,$ where $T_\theta$ is the bundle
isomorphism of Lemma \ref{antistoixo}. The complex valued functions $H_3, H_4, H_5, H_6$ of $g,$ associated 
to the orthonormal frame $\{e_1,e_2,e_3,e_4,e_5,e_6\}$ and the corresponding
functions $H_3^\theta, H_4^\theta, H_5^\theta, H_6^\theta$ of $g_\theta$, associated to the orthonormal
frame
$\{e_1,e_2,T_\theta e_3,T_\theta e_4,T_\theta e_5,T_\theta e_6\}$ satisfy
\begin{equation}\label{xrhsimoweing}
H_3^\theta=e^{-i\theta}H_3,\,\,H_4^\theta=e^{-i\theta}H_4,\,\,H_5^\theta=e^{-i\theta}H_5\,\text{ and }H_6^\theta=e^{-i\theta}H_6.
\end{equation}

Using \eqref{xrhsimoweing} and the Weingarten formula for $g_\theta,$ we obtain 
\begin{equation}\label{weine3}
\tilde{\nabla}_ET_{\theta}e_3=\omega_{34}(E)T_{\theta}e_4+\frac{\kappa_2}{\kappa_1}T_{\theta}e_5-\frac{i\mu_2}{\kappa_1}T_{\theta}e_6-\kappa_1 e^{i\theta}dg_{\theta}(\overline{E}),
\end{equation}
\begin{equation}\label{weine4}
\tilde{\nabla}_ET_{\theta}e_4=-\omega_{34}(E)T_{\theta}e_3+\frac{i\kappa_2}{\kappa_1}T_{\theta}e_5+\frac{\mu_2}{\kappa_1}T_{\theta}e_6+i\kappa_1 e^{i\theta}dg_{\theta}(\overline{E}),
\end{equation}
\begin{equation}\label{weine5}
\tilde{\nabla}_ET_{\theta}e_5=\omega_{56}(E)T_{\theta}e_6-\frac{\kappa_2}{\kappa_1}\left(T_{\theta}e_3+iT_{\theta}e_4\right),
\end{equation}
\begin{equation}\label{weine6}
\tilde{\nabla}_ET_{\theta}e_6=-\omega_{56}(E)T_{\theta}e_5+\frac{i\mu_2}{\kappa_1}\left(T_{\theta}e_3+iT_{\theta}e_4\right),
\end{equation}
where $E=e_1-ie_2$ and $\tilde{\nabla}$ stands for the usual connection in the induced bundle $(i_1\circ f)^*(T\mathbb{R}^7),$ with $i_1\colon \mathbb{S}^6\to\mathbb{R}^7$ being the inclusion map.

\begin{lemma}
Suppose that for $\theta_j\in\mathcal{M}^K_6(g), j=1,\dots,m,$ there exist vectors $v_j\in\mathbb{R}^7,$ such that 
\begin{equation*}
\sum\limits_{j=1}^m\langle g_{\theta_j},v_j\rangle=0 \,\,\text{ on }\,\, U.
\end{equation*}
Then the following hold:
\begin{equation}\label{seconduseful}
\sum\limits_{j=1}^me^{i\theta_j}\left(\kappa_2\langle T_{\theta_j}e_5,v_j\rangle-i\mu_2\langle T_{\theta_j}e_6,v_j\rangle\right)=0,
\end{equation}
away from the zeros of $\omega_{56},$ and
\begin{equation}\label{thirduseful}
\overline{E}\Big(\sum\limits_{j=1}^me^{i\theta_j}\langle T_{\theta_j}e_6,v_j\rangle\Big)=-\omega_{56}(\overline{E})\sum\limits_{j=1}^me^{i\theta_j}\langle T_{\theta_j}e_5,v_j\rangle.
\end{equation}
\end{lemma}
\begin{proof}
Our assumption implies that 
\begin{equation*}
\sum\limits_{j=1}^m\langle dg_{\theta_j},v_j\rangle=0.
\end{equation*}
Differentiating and using the Gauss formula we obtain
\begin{equation}\label{auseful}
\sum\limits_{j=1}^me^{i\theta_j}\langle T_{\theta_j}e_3-iT_{\theta_j}e_4,v_j\rangle=0.
\end{equation}

Differentiating \eqref{auseful} with respect to $E$ and using \eqref{xrhsimoweing}, \eqref{weine3} and \eqref{weine4}, it follows that
$$
\sum\limits_{j=1}^me^{i\theta_j}\left(\overline{H}_5\langle T_{\theta_j}e_5,v_j\rangle+\overline{H}_6\langle T_{\theta_j}e_6,v_j\rangle\right)=0.
$$
Using that $H_5=\kappa_2$ and $H_6=i\mu_2$ (see Lemma \ref{ksanauseful}(ii)), the above yields \eqref{seconduseful}.

From \eqref{weine6}, we compute that 
\[
\overline{E}\Big(\sum\limits_{j=1}^me^{i\theta_j}\langle T_{\theta_j}e_6,v_j\rangle\Big)=-\omega_{56}(\overline{E})\sum\limits_{j=1}^me^{i\theta_j}\langle T_{\theta_j}e_5,v_j\rangle -\frac{i\mu_2}{\kappa_1}\sum\limits_{j=1}^me^{i\theta_j}\langle T_{\theta_j}e_3-iT_{\theta}e_4,v_j\rangle,
\]
which in view of \eqref{auseful} yields \eqref{thirduseful}.
\end{proof}

We recall the following result \cite{V16}.
\begin{lemma}\label{apoesch}
Let $f\colon M\to\mathbb{S}^n$ be a compact exceptional surface. The Euler-Poincar\'{e}
number $\chi(N_r^fM)$ of the $r$-th normal bundle and the Euler-Poincar\'{e}
characteristic $\chi(M)$ of $M$ satisfy the following:

(i) If $\Phi_r\neq0$ for some $1\le r<m,$ where $m=[(n-1)/2],$ then
\[
\chi(N_r^fM)=0 \text{\,\, and \,\,} (r+1)\chi(M)=-N(a_r^+)=-N(a_r^-).
\]

(ii) If $\Phi_r=0$ for some $1\le r\le m,$ then
\[
(r+1)\chi(M)-\chi(N_r^fM)=-N(a_r^+).
\]

(iii)If $\Phi_m\neq0,$ then
 \[
(m+1)\chi(M)\mp\chi(N_m^fM)=-N(a_m^\pm).
\]
\end{lemma}
Now we are able to prove Theorem \ref{compnon}.
\begin{proof}[Proof of Theorem \ref{compnon}]
According to Theorem \ref{finiteorcircle}, the space $\mathcal{M}^K_6(g)$ of the
isometric deformations that are isometric to $g$ is either $[0,\pi)$ or a finite subset of $[0,\pi).$ 
Suppose to the contrary that $\mathcal{M}^K_6(g)=[0,\pi).$ We claim that the
coordinate functions of the minimal surfaces $g_\theta, \theta\in[0,\pi),$ are linearly independent. Since these functions are eigenfunctions of the Laplace operator of $M$ with corresponding eigenvalue 2, this contradicts the fact that the eigenspaces of the Laplace operator are finite dimensional. To prove that the coordinate functions are linearly independent, it is enough to prove that if 
\begin{equation}\label{firstusefulagain}
\sum\limits_{j=1}^m\langle g_{\theta_j},v_j\rangle=0,
\end{equation}
for $0<\theta_1<\cdots<\theta_m<\pi,$ then $v_j=0$ for all $1\le j\le m.$

Assume to the contrary that $v_j\neq0$ for all $1\le j\le m.$ Let $M_1=\{p_1,\dots,p_k\}$ be the zero set of $\Phi_2.$ Around each point $p\in M\smallsetminus M_1,$ we choose local complex coordinate $(U,z)$ and an orthonormal frame $\{e_1,e_2,e_3,e_4,e_5,e_6\}$ on $U\subset M\smallsetminus M_1$ as in Lemma \ref{ksanauseful}. We consider the complex valued function
\[
\psi:=\Big(\sum\limits_{j=1}^me^{i\theta_j}\langle T_{\theta_j}e_6,v_j\rangle\Big)^2,
\]
where $T_{\theta_j}\colon N_gM\to N_{g_{\theta_j}}M$ is the bundle
isomorphism of Lemma \ref{antistoixo}. Obviously $\psi$ is well defined on $M\smallsetminus M_1.$ Equations \eqref{connectionforms} imply that
\begin{equation*}
E(\kappa_2)=i\mu_2\omega_{56}(E)-3i\kappa_2\omega_{12}(E),
\end{equation*}
and
\begin{equation*}
E(\mu_2)=i\kappa_2\omega_{56}(E)-3i\mu_2\omega_{12}(E).
\end{equation*}
These yield
\[
\omega_{56}(\overline{E})=\frac{i}{\kappa_2^2-\mu_2^2}\left(\kappa_2\overline{E}(\mu_2)-\mu_2\overline{E}(\kappa_2) \right).
\]
Then \eqref{seconduseful} and \eqref{thirduseful} imply that $\overline{E}\big(\psi\left(1-\mu_2^2/\kappa_2^2\right)\big)=0,$ and hence the function $\Psi:=\psi\left(1-\mu_2^2/\kappa_2^2\right)\colon M\smallsetminus M_1\to\mathbb{C}$ is holomorphic. Since $\Psi$ is bounded, its isolated singularities are removable and consequently there exists a constant $c$ such that
\begin{equation}\label{limit}
\psi(\kappa_2^2-\mu_2^2)=c\kappa_2^2\,\,\text{ on\,}\, M\smallsetminus M_1.
\end{equation}

We claim that $c=0.$ Indeed, if $\kappa_2(p_l)=\mu_2(p_l)>0$ for some $1\le l\le k,$ then taking the limit in \eqref{limit} along a sequence of points in $M\smallsetminus M_1$ that converges to $p_l,$ we deduce that $c=0.$

Suppose now that $\kappa_2(p_l)=\mu_2(p_l)=0$ for all $1\le l\le k.$ Let $(V,z)$ be a local complex coordinate around $p_l$ with $z(p_l)=0.$ 
From the proof of \cite[Proposition 4]{V08} for $s=2$
we obtain 
\[
d\overline{H}_5-3i\overline{H}_5\omega_{12}-\overline{H}_6\omega_{56}\equiv 0\;\mathrm{mod}\; \phi, 
\]
and
\[
d\overline{H}_6-3i\overline{H}_6\omega_{12}+\overline{H}_5\omega_{56}\equiv 0\;\mathrm{mod}\; \phi.
\]
Writing $\phi=Fdz,$ we deduce that 
\[
\frac{\partial\overline{H}_5}{\partial\overline{z}}=3i\overline{H}_5\omega_{12}(\overline{\partial})+\overline{H}_6\omega_{56}(\overline{\partial})
\]
and
\[
\frac{\partial\overline{H}_6}{\partial\overline{z}}=3i\overline{H}_6\omega_{12}(\overline{\partial})-\overline{H}_5\omega_{56}(\overline{\partial}).
\]
Using a theorem due to Chern \cite[p. 32]{Ch}, we may write 
\[
\overline{H}_5=z^{m_l}H^*_5\text{\,\,\,\, and \,\,\,}\overline{H}_6=z^{m_l}H^*_6,\]
where $m_l$ is a positive integer and $H^*_5,H^*_6$ are nonzero smooth complex functions.
Since 
$$
\alpha_3(E,E,E)=4(\overline{H}_5e_5+\overline{H}_6e_6),
$$ 
we obtain
\begin{equation}\label{a3tilde}
\alpha_{3}^{(3,0)}=z^{m_l}\alpha_{3}^{*(3,0)}\,\,\text{ on }\, V,
\end{equation}
where $\alpha_{3}^{*(3,0)}$ is a tensor field of type $(3,0)$ with $\alpha_{3}^{*(3,0)}|_{p_l}\neq0.$
We now define the $N_2^g$-valued tensor field $\alpha_{3}^{*}:=\alpha_{3}^{*(3,0)}+\overline{\alpha_{3}^{*(3,0)}}.$ It is clear that $\alpha_{3}^*$
maps the unit circle on each tangent plane into an ellipse, whose length of the semi-axes are denoted by $\kappa^*_2\ge
\mu^*_2\ge0.$ We furthermore consider the differential form of type $(6,0)$ 
\[
\Phi^*_2:=\langle \alpha_{3}^{*(3,0)},\alpha_{3}^{*(3,0)}\rangle dz^6,
\]
which in view of \eqref{a3tilde}, is related to the Hopf differential of $g$ by $\Phi_2=z^{2m_l}\Phi_2^*.$
We split $\Phi_2$ and $\Phi_2^*,$ with respect to an arbitrary orthonormal frame $\{\xi_1,\dots\xi_6\},$ where $\{\xi_1,\xi_2\}$
and $\{\xi_5,\xi_6\}$ are arbitrary orthonormal frames  of $TV$  and $N^g_2V$ respectively as 
\[
\Phi_2=\frac{1}{4}\big(\overline{H}_5^2+\overline{H}_6^2\big)\phi^6=\frac{1}{4}k^{+}_{2}k^{-}_{2}\phi^6,
\]
\[
\Phi^{*}_{2}=\frac{1}{4}\big(\overline{H}^{*2}_5+\overline{H}^{*2}_6
\big)\phi^6=\frac{1}{4}{k}^{*+}_{2}{k}^{*-}_{2}\phi^6,
\]
where $k^{\pm}_{2}=\overline{H}_5\pm i\overline{H}_6, k^{*\pm}_{2}=\overline{H}^*_5\pm i\overline{H}^*_6,$
\[
{H}^*_5=\langle
{\alpha}^*_{3}(e_1,e_1,e_1),e_5\rangle+i\langle
{\alpha}^*_{3}(e_1,e_1,e_2),e_5\rangle
\]
and
\[
{H}^*_6=\langle
{\alpha}^*_{3}(e_1,e_1,e_1),e_6\rangle+i\langle
{\alpha}^*_{3}(e_1,e_1,e_2),e_6\rangle.
\]
From \eqref{a3tilde}, we obtain $\overline{H}_{a}=z^{m_l}\overline{H}^*_{a}$ for $a=5,6,$ or equivalently,
$k_2^{\pm}=z^{m_l}k_2^{*\pm}.$ Observe that $a_2^{\pm}=|k_2^{\pm}|.$ Hence
\begin{equation}\label{kappamu}
\kappa_2=|z|^{m_l}{\kappa}^*_2,\,\, \mu_2=|z|^{m_l}{\mu}^*_2.
\end{equation}
Now \eqref{limit} yields
\begin{equation}\label{limitnew}
\psi({\kappa}_2^{*2}-{\mu}_2^{*2})=c{\kappa}_2^{*2}\,\,\text{ on\,}\, V\smallsetminus \{p_l\}.
\end{equation}

If ${\kappa}^*_2(p_l)>{\mu}^*_2(p_l)$ for all $1\le l\le k,$ then \eqref{kappamu} implies that 
$$N(a^+_2)=\sum\limits_{l=1}^{k}m_l=N(a^-_2).$$
Hence, Lemma \ref{apoesch} yields $\mathcal{\chi}(N_2^f)=0,$ which contradicts our assumption. Thus, ${\kappa}^*_2(p_l)={\mu}^*_2(p_l)$ for some $1\le l\le k.$ Taking the limit in \eqref{limitnew}, along a sequence of points in $V\smallsetminus \{p_l\}$ which converges to $p_l,$ we obtain $c{\kappa}_2^{*2}(p_l)=0.$ Since ${\alpha}^*_3|_{p_l}\neq0,$ we derive that $c=0.$

In view of \eqref{limit}, we conclude that $\psi=0$ on $M\smallsetminus M_1.$ This implies that 
\[
\sum\limits_{j=1}^me^{i\theta_j}\langle T_{\theta_j}e_6,v_j\rangle=0,
\]
 which due to \eqref{seconduseful} gives that 
\begin{equation*}
\sum\limits_{j=1}^me^{i\theta_j}\langle T_{\theta_j}e_5,v_j\rangle=0.
\end{equation*}
Differentiating this with respect to $E,$ and using \eqref{weine5} and the above, we obtain 
\begin{equation*}
\sum\limits_{j=1}^me^{i\theta_j}\langle T_{\theta_j}e_3+iT_{\theta_j}e_4,v_j\rangle=0
\end{equation*}
which combined with \eqref{auseful} yields
\begin{equation}\label{more}
\sum\limits_{j=1}^me^{i\theta_j}\langle T_{\theta_j}e_3,v_j\rangle=0=
\sum\limits_{j=1}^me^{i\theta_j}\langle T_{\theta_j}e_4,v_j\rangle.
\end{equation}
Differentiating \eqref{more}  with respect to $E$ we find that
\begin{equation*}
\sum\limits_{j=1}^me^{2i\theta_j}\langle dg_{\theta_j}(\overline{E}),v_j\rangle=0.
\end{equation*}
Differentiating once more with respect to $E$ and using the minimality of each $g_{\theta_j}$we obtain 
\begin{equation*}
\sum\limits_{j=1}^me^{2i\theta_j}\langle g_{\theta_j},v_j\rangle=0.
\end{equation*}
Combining this with \eqref{firstusefulagain}, we obtain
\begin{equation*}
\sum\limits_{j=2}^m\langle g_{\theta_j},w_j\rangle=0,
\end{equation*}
where $w_j:=\lambda_jv_j\neq0, j=2,\dots,m$ and $\lambda_j=\cos2\theta_m-\cos2\theta_1$ or $\lambda_j=\sin2\theta_m-\sin2\theta_1.$ By induction, we finally conclude that $\langle g_{\theta_m},w\rangle=0,$ for some nonzero vector $w.$ Therefore, $g_{\theta_m}$ lies in a totally geodesic $\mathbb{S}^5,$ which is a contradiction and the theorem is proved.
\end{proof}

\begin{remark}\label{remarknoniso}
The global assumptions and the assumption on the codimension in Theorem \ref{compnon} are essential and can not
be dropped. In fact, locally we can produce
minimal surfaces in spheres that are isometric to a nonisotropic pseudoholomorphic curve $g$
in $\mathbb{S}^6.$ More precisely, let 
$g_{\theta}, 0\leq \theta<\pi$, be the associated family of a simply connected nonisotropic
pseudoholomorphic curve $g\colon M\to\mathbb{S}^{6}.$ We consider the surface
$G\colon M\to\mathbb{S}^{7m-1}$ defined by
\begin{equation*}
G=a_{1}g_{\theta _{1}}\oplus \cdots\oplus a_{m}g_{\theta _{m}},  
\end{equation*}
where $a_{1},\dots\,,a_{m}$ are any real numbers with $\sum_{j=1}^{m}a_{j}^
{2}=1,$ $0\leq \theta _{1}<\cdots<\theta_{m}<\pi,$ and $\oplus $ denotes the
orthogonal sum with respect to an orthogonal decomposition of the Euclidean space
$\mathbb{R}^{7m}.$  Arguing as in \cite{VT}, it is easy to see that the surface $G$ is minimal and isometric to $g.$ 
\end{remark}

\begin{proposition}\label{1iso}
Let $g\colon M\to\mathbb{S}^6$ be a compact nonisotropic and substantial pseudoholomorphic curve.
If $\hat{g}\colon M\to\mathbb{S}^n$
is a minimal surface that is isometric to $g,$ then $\hat{g}$ is 1-isotropic.
\end{proposition}
\begin{proof}
According to \cite[Theorem 2]{V16}, the function $1-K$ is of absolute value type.
If the zero set of the function $1-K$ is empty, then from the condition \eqref{trik} it follows that $M$
is homeomorphic to the sphere. From \cite{Calabi} we have that $\hat{g}$ is isotropic and from
\cite{V} it follows that $n=6$ and $\hat{g}$ is congruent to $g.$
Now suppose that the zero set of the function $1-K$ is the nonempty set
$M_0=\left\{ p_1,\dots,p_m\right\} $ with corresponding order $\mathrm{ord}_{p_j}(1-K)=2k_j.$
For each point $p_j, j=1,\dots,m,$ we choose a
local complex coordinate $z$ such that $p_j$ corresponds to $z=0$\ and the induced
metric is written as $ds^2=F|dz|^2.$ Around\ $p_j,$ we have that
\begin{equation}\label{onemoreavt}
1-K=|z|^{2k_j}u_0,
\end{equation}
where $u_0$ is a smooth positive function.

We know that the first Hopf differential $\hat{\Phi}_1=\hat{f}_1dz^4$ of $\hat{g}$ is globally defined
and holomorphic. We claim that $\hat{\Phi}_1$ is identically zero. We assume to
the contrary that it is not identically zero. Hence its zeros are isolated. 
Each $p_j$ is totally geodesic, according to \eqref{Gausseqa2}, and obviously, $\hat{\Phi}_1$ vanishes at each $p_j.$
Thus we may write $\hat{f}_1=z^{l_{1j}}\psi _1$\ around\ $p_j,$ where
$l_{1j}$\ is the order of $\hat{\Phi}_1$\ at $p_j,$ and $\psi _1$ is a nonzero
holomorphic function. Bearing in mind (\ref{what}),\ we obtain 
\begin{equation}\label{a2gias6}
\frac{1}{4}\left\Vert \hat{\alpha}_2\right\Vert^4-(\hat{K_1}^{\perp})^2=2^4F^{-4}|\psi _1|^2|z|^{2l_{1j}}
\end{equation}
around $p_j,$ where $ \hat{\alpha}_2$ and $\hat{K_1}^{\perp}$ are respectedly the second fundmental form and the first normal curvature of $\hat{g}.$ We now consider the function $u_1\colon M
\smallsetminus M_0\to \mathbb{R}$ defined by
$$
u_1=\frac{\frac{1}{4}\left\Vert \hat{\alpha}_2\right\Vert^4-(\hat{K_1}^{\perp})^2}{(1-K)^2}.
$$
In view of (\ref{onemoreavt}) and \eqref{a2gias6} we have that
\[
u_1=2^{4}F^{-4}u_0^{-2}|\psi _1|^2|z|^{2(l_{1j}-2k_j)}.
\]

Using \eqref{Gausseqa2} we find that $u_1\leq 1,$ thus from the above and \eqref{onemoreavt} we deduce that $l_{1j}\geq 2k_j$. Hence we
can extend $u_1$ to a smooth function on $M.$ Applying Proposition \ref{3i}(i) for $s=1$ for $g$
and Proposition \ref{3i}(ii) for $r=1$ for $\hat{g}$ we have
that
\begin{equation*}
\Delta \log \left\Vert \alpha_{2}\right\Vert ^2=2\big(2K-K_1^{\ast}\big),
\end{equation*}

\begin{equation*}
\Delta \log \left(\left\Vert \hat{\alpha}_{2}\right\Vert^2+2\hat{K}_1^{\perp}\right) =2\big(2K-\hat{K}_1^{\ast}\big)
\end{equation*}
and 
\begin{equation*}
\Delta \log \left(\left\Vert \hat{\alpha}_{2}\right\Vert^2-2\hat{K}_1^{\perp}\right) =2\big(2K+\hat{K}_1^{\ast}\big).
\end{equation*}
Combining these equation we obtain 
\begin{equation}\label{ahatnew}
\Delta\log u_1=4K_1^*,
\end{equation}
away from the isolated zeros of $u_1,$ where
$K_1^*$ is the intrinsic curvature of the first normal bundle $N_1^g.$
Moreover Proposition \ref{3i}(iii) for $r=1,$ in combination with \eqref{trik} provides
\[
K-\frac{1}{2}<-K_1^*<K,
\]
or more specific
\[
K_1^*+K>0.
\]
Hence, \eqref{ahatnew} yields that $\Delta\log u_1+4K>0$ and consequently using Lemma \ref{forglobal} and the Gauss-Bonnet theorem
we have that
\begin{equation*}
N(u_1)\le4\chi (M)\le0 ,
\end{equation*}
where $\chi (M)$ is the Euler-Poincar\'{e} characteristic of $M.$
This implies that $N(u_1)=0,$ which contradicts our assumption that $\Phi_1=0.$
\end{proof}

In view of Proposition \ref{1iso}, the surface $\hat{g}$ in Theorem \ref{teleutaioPhi} is $1$-isotropic and consequently the Hopf differential $\hat{\Phi}_2$ of $\hat{g}$ is not identically zero.
The following lemma will be used for the proof of Theorem \ref{teleutaioPhi}.

\begin{lemma}\label{lemma9}
Under the assumptions of Theorem \ref{teleutaioPhi}, the following assertions hold:

(i) The $a$-invariants of $g$ and $\hat{g}$ satisfy the inequality
\[
a_2^- \hat{a}_2^+\le a_2^+\hat{a}_2^-.
\]

(ii) The eccentricities $\varepsilon_2, \hat{\varepsilon}_2$ of the second curvature
ellipses of $g$ and $\hat{g}$ respectively satisfy the inequality $\varepsilon_2\le\hat{\varepsilon}_2.$

(iii) There exists a constant $c\ge1$ such that the lengths $\kappa_2,\mu_2$ and
$\hat{\kappa}_2,\hat{\mu}_2$ of the semi-axes of the second curvature ellipses of
the surfaces $g$ and $\hat{g}$ respectively satisfy 
\begin{equation}\label{peris}
\kappa_2^2-\mu_2^2=c(\hat{\kappa}_2^2-\hat{\mu}_2^2).
\end{equation}

(iv) At a point $p\in M,$ we have that $a_2^+(p)=0$
if and only if $\hat{a}_2^+(p)=0.$ 

(v) If $\hat{a}_2
^+(p)>0$ at a point $p\in M,$ then $\hat{a}_2^-(p)=0$ if and only if $a_2^-(p)=0.$
\end{lemma}
\begin{proof}
(i) It follows from Proposition \ref{1iso}, Propositions \ref{5} and \ref{3i}  and the Gauss
equation that
$\left\Vert \hat{\alpha}_{3}\right\Vert=\left\Vert\alpha_{3}\right\Vert,$ where $\hat{\alpha}_{3}$ is the third fundamental form of $\hat{g}.$ This means
that 
\begin{equation}\label{meiwsh}
\hat{\kappa}_2^2+\hat{\mu}_2^2=\kappa_2^2+\mu_2^2.
\end{equation}
Combining the above with our assumption $\hat{\kappa}_2\hat{\mu}_2\le\kappa_2\mu_2,$ we have that
$$\hat{\kappa}_2+\hat{\mu}_2\le\kappa_2+\mu_2\text{\,\,\ and\,\,}\kappa_2-\mu_2\le\hat{\kappa}_2-\hat{\mu}_2.$$
The proof of part (i) follows easily.

(ii) Since $\hat{K}_2^\perp\le K_2^\perp,$ equation \eqref{meiwsh} implies that
\[
\frac{\hat{\kappa}_2\hat{\mu}_2}{\hat{\kappa}_2^2+\hat{\mu}_2^2}\le
\frac{\kappa_2\mu_2}{\kappa_2^2+\mu_2^2}.
\]
We set  $\hat{t}_2:=\hat{\mu}_2/\hat{\kappa}_2$ and $t_2:=\mu_2/\kappa_2.$ Obviously, 
$0\le\hat{t}_2,t_2\le1$ and
\[
\frac{\hat{t}_2}{1+\hat{t}_2^2}\le\frac{t_2}{1+t_2^2}.
\]

This immediately implies that $\varepsilon_2\le\hat{\varepsilon}_2.$

(iii) From Proposition \ref{3i}(ii) we have that 

\begin{equation}\label{prwti}
\Delta\log(\kappa_2+\mu_2)=3K-K_2^*,
\,\,\,\ 
\Delta\log(\kappa_2-\mu_2)=3K+K_2^*,
\end{equation}
and 
\begin{equation}\label{tetarti}
\Delta\log(\hat{\kappa}_2+\hat{\mu}_2)=3K-\hat{K}_2^*,
\,\,\,\ 
\Delta\log(\hat{\kappa}_2-\hat{\mu}_2)=3K+\hat{K}_2^*,
\end{equation}
where $\hat{K}_2^*$ denotes the second intrinsic curvature of $\hat{g}.$
Equations \eqref{prwti} and \eqref{tetarti} imply that
\begin{equation*}
\Delta\log\left(\left\Vert \alpha_{3}\right\Vert^4-16(K_2^{\perp})^2\right)=12K
\text{\,\,and\, \,}
\Delta\log\big(\left\Vert \hat{\alpha}_{3}\right\Vert^4-16(\hat{K}_2^{\perp})^2\big)=12K.
\end{equation*}
Inequality $\hat{K}_2^\perp\le K_2^\perp$ yields 
\begin{equation}\label{flast}
|f_2|^2\le|\hat{f}_2|^2,
\end{equation}
where $\Phi_2=f_2dz^6$ and $\hat{\Phi}_2=\hat{f}_2dz^6.$ For each point $p_j\in M_0
=\{p_1,\dots,p_m\}, j=1,\dots,m,$ where $M_0$ is the union of the zero sets of $\Phi_2$ and $\hat\Phi_2,$
we choose a local complex coordinate $z$ such that
$p_j$ corresponds to $z=0$ and the induced metric is written as $ds^2=F|dz|^2.$

Suppose that $\hat{\Phi}_2(p_j)=0$ for some $j=1,\dots,m.$ Then Lemma
\ref{lemma9}(ii) implies that $\Phi_2(p_j)=0.$
Thus we may write $f_2=z^{m(p_j)}u$ and $\hat{f}_2=z^{\hat{m}(p_j)}\hat{u}$
around $p_j,$ where $m(p_j)$ and $\hat{m}(p_j)$ are the orders of $\Phi_2$ and $\hat{\Phi}_2$
respectively at $p_j$ and $u$ and $\hat{u}$ are nonzero holomorphic functions. From
\eqref{flast} we have that $\hat{m}\le m,$ and therefore the function $u_2=|f_2|^2/|\hat{f}_2|^2
\colon M\smallsetminus M_0\to \mathbb{R}$ can be extended to a smooth
function on $M.$ 

Suppose now that $\hat{\Phi}_2(p_j)\neq0$ for some $j=1,\dots,m.$ We have that the function $u_2=|z|^{2m(p_j)}u,$
with $u$ a positive smooth function, can be extended to a smooth
function on $M.$ 

In both cases we have that the function $u_2$ is subharmonic and the maximum principle
yields \eqref{peris}. Obviously \eqref{peris} gives that the zeros of the second Hopf differential $\Phi_2$ of the curve $g$ coincide with
the zeros of the second Hopf differential $\hat{\Phi}_2$ of the surface $\hat{g}.$

(iv) If $a_2^+(p)=0$ at a point $p\in M,$ we obtain $\kappa_2(p)=
\mu_2(p)=0.$ It follows from \eqref{meiwsh} that $\hat{\kappa}_2(p)=\hat{\mu}_2
(p)=0,$ which is $\hat{a}_2^+(p)=0.$ 

(v) Part (v) follows immediately from (i) and \eqref{peris} which is equivalently written as
$a_2^+a_2^-=c \hat{a}_2^+\hat{a}_2^-.$
\end{proof}
Now we prove Theorem \ref{teleutaioPhi}.
\begin{proof}[Proof of Theorem \ref{teleutaioPhi}]
Equations \eqref{prwti} and \eqref{tetarti} yield
\begin{equation}\label{lastnoleast}
\Delta\log\frac{a_2^-\hat{a}_2^+}{a_2^+\hat{a}_2^-}
=2(K^*_2-\hat{K}_2^*),
\end{equation}
on $M\smallsetminus M_0,$ where $M_0=\{p_1,\dots,p_m\}$ is the
union of the zero sets of $\Phi_2$ and $\hat\Phi_2.$
For each point $p_j\in M_0=\{p_1,\dots,p_m\}, j=1,\dots,m,$ we choose a local
complex coordinate $z$ such that $p_j$ corresponds to $z=0$ and the induced
metric is written as $ds^2=F|dz|^2.$

We now claim that the function $u=(a_2^-\hat{a}_2^+)/(a_2^+\hat{
a}_2^-)\colon M\smallsetminus M_0\to \mathbb{R}$ can be extended
to a smooth function on $M.$ To this aim we distinguish the following cases: 

\textit {Case I:} Suppose that $\hat{a}_2^+(p_j)=0$ for some $j=1,\dots,m.$ Then Lemma
\ref{lemma9}(iv) implies that $a_2^+(p_j)=0.$ Hence $\hat{a}
_2^-(p_j)=a_2^-(p_j)=0.$ The $a$-invariants are absolute value type functions,
thus we may write $a_2^+=|z|^{2m_+}u_+,$ $a_2^-=|z|^{2m_-}u_-,$ $\hat{a}
_2^+=|z|^{2\hat{m}_+}\hat{u}_+$ and $\hat{a}_2^-=|z|^{2\hat{m}_-}\hat{u}_-$
around $p_j,$ where $m_+,m_-,\hat{m}_+$ and $\hat{m}_-$ are the orders of $a_2^+,$
$a_2^-,$ $\hat{a}_2^+$ and $\hat{a}_2^-$ respectively at $p_j$
and $u_+, u_-, \hat{u}_+$ and $\hat{u}_-$ are nonvanishing smooth functions.
From Lemma \ref{lemma9}(i) it follows that 
\[
m_-(p_j)+\hat{m}_+(p_j)\ge m_+(p_j)+\hat{m}_-(p_j).
\]
Therefore the function $u=(a_2^-\hat{a}_2^+)/(a_2^+\hat{
a}_2^-)$ can be extended
to a smooth function around $p_j.$ 

\textit{Case II:} Suppose that $\hat{a}_2^+(p_j)>0$ for some $j=1,\dots,m.$
Lemma \ref{lemma9}(v) implies that either $\hat{a}_2^-(p_j)a_2^-(p_j)>0$ or $\hat{a}_2^-
(p_j)=a_2^-(p_j)=0.$ In the former case, by Lemma \ref{lemma9}(i) we have that $a_2^+(p_j)>0.$
Thus $u$ is well defined at $p_j$. 

Now assume that $\hat{a}_2^-
(p_j)=a_2^-(p_j)=0.$ Clearly \eqref{meiwsh} implies that $a_2^+(p_j)>0.$ Since the $a$-invariants are absolute value type functions, we may write $a_2^-=|z|^{2m_-}u_-$ and
$\hat{a}_2^-=|z|^{2\hat{m}_-}\hat{u}_-$ around $p_j,$ where $m_-$ and
$\hat{m}_-$ are the orders of $a_2^-,$ and $\hat{a}_2^-$ respectively at $p_j$
and $u_-$ and $\hat{u}_-$ are nonvanishing smooth functions.
Lemma \ref{lemma9}(i) yields
\[
m_-(p_j)\ge\hat{m}_-(p_j),
\]
therefore the function $u=(a_2^-\hat{a}_2^+)/(a_2^+\hat{
a}_2^-)\colon M\smallsetminus M_0\to \mathbb{R}$ can be extended
to a smooth function around $p_j.$ 


It follows from Proposition \ref{5}
and \eqref{lastnoleast} that 
\begin{equation}\label{least}
\Delta\log u=\frac{2\left\Vert \alpha_{2}\right\Vert^2}{(K_1^\perp)^2}
(K_2^{\perp}-\hat{K}_2^{\perp})+\frac{2\left\Vert \hat{\alpha}_{4}\right\Vert^2}
{4\hat{K}_2^{\perp}}
\end{equation}
away from the isolated zeros of $u.$ Hence $\Delta\log u\ge0$ on $M\smallsetminus M_0.$
By continuity, the function $u$ is subharmonic on $M$ 
and from the maximum principle we have that $u$ is constant. Then from
\eqref{least} it follows that  $\hat{K}_2^\perp= K_2^\perp,$ and $\hat{\alpha}_{4}=0.$
Hence $f(M)$ is contained in a totally geodesic sphere $\mathbb{S}^6$ in
$\mathbb{S}^n$.

The fact that the set of all noncongruent minimal surfaces $\hat{g},$ as in the statement
of the theorem, that are isometric to $g$ is either a circle or a finite set, follows directly
from Theorem \ref{finiteorcircle}.
\end{proof}
\begin{corollary}
Let $g\colon M\to\mathbb{S}^6$ be a compact nonisotropic and substantial pseudoholomorphic
curve with second normal curvature $K_2^\perp.$ Any
substantial minimal surface $\hat{g}$ in $\mathbb{S}^n, n>6,$ whose second normal curvature
$\hat{K}_2^\perp$ satisfies the inequality $\hat{K}_2^\perp\le K_2^\perp,$
cannot be isometric to $g.$
\end{corollary}
\begin{proof}
Assume that $\hat{g}$ is isometric to $g$. Proposition \ref{1iso} implies that $\hat{g}$ is 1-isotropic.
Suppose that $n>6.$ Then Theorem \ref{teleutaioPhi} implies that $\hat{g}$ is 2-isotropic.
Hence $\hat{\kappa}_2=
\hat{\mu}_2.$ The inequality $\hat{K}_2^\perp\le K_2^\perp,$ in combination
with \eqref{meiwsh} implies that $\kappa_2=\mu_2,$ which is a contradiction.
\end{proof}

\end{document}